\documentclass{amsart}
\usepackage{color}
\usepackage{graphicx}
\usepackage{hyperref}
\usepackage{amssymb}
\usepackage{amsfonts}
\usepackage{euscript}
\usepackage[all]{xy}
\usepackage{mathabx}
\usepackage{mathtools}
\usepackage{bbm}
\usepackage{lscape}
\usepackage{enumitem}

\numberwithin{equation}{section}
\numberwithin{figure}{section}
\renewcommand{\subsection}{\hspace{-\parindent}\refstepcounter{subsection}{\bf \arabic{section}\alph{subsection}. }}

\theoremstyle{plain}
\newtheorem{thm}{Theorem}[section]
\newtheorem{theorem}[thm]{Theorem}

\newtheorem{corollary}[thm]{Corollary}

\newtheorem{assumption}[thm]{Assumption}
\newtheorem{definition}[thm]{Definition}

\newtheorem{remark}[thm]{Remark}

\newtheorem{proposition}[thm]{Proposition}
\newtheorem*{proposition*}{Proposition}
\newtheorem{example}[thm]{Example}
\newtheorem{non-example}[thm]{Non-example}

\newtheorem{lemma}[thm]{Lemma}

\newtheorem*{claim*}{Claim} 
\newtheorem*{lemma*}{Lemma}
\newtheorem{point}{Point}

\newtheorem*{theorem*}{Theorem}
\newtheorem*{result*}{Result}
\newtheorem*{conjecture*}{Conjecture}
\newtheorem{Question}{Question}

\newtheorem{task}{Task}

\newcommand{\bC}{{\mathbb C}}

\newcommand{\bF}{{\mathbb F}}

\newcommand{\bK}{{\mathbb K}}

\newcommand{\bR}{{\mathbb R}}

\newcommand{\bZ}{{\mathbb Z}}

\title{A survey of equivariant operations on quantum cohomology for symplectic manifolds}

\author{Nicholas Wilkins}
\date{\today}

\begin{document}

\maketitle

\begin{abstract}
    In this survey paper, we will collate various different ideas and thoughts regarding equivariant operations on quantum cohomology (and some in more general Floer theory) for a symplectic manifold. We will discuss a general notion of equivariant quantum operations associated to finite groups, in addition to their properties, examples, and calculations. We will provide a brief connection to Floer theoretic invariants. We then provide abridged descriptions (as per the author's understanding) of work by other authors in the field, along with their major results. Finally we discuss the first step to compact groups, specifically $S^1$-equivariant operations. Contained within this survey are also a sketch of the idea of mod-$p$ pseudocycles, and an in-depth appendix detailing the author's understanding of when one can define these equivariant operations in an additive way.\end{abstract}

\section*{Acknowledgements}{{\it Thank you for useful conversations to Alex Ritter, Paul Seidel and Nate Bottman, mainly for ideas on structural things, operadic things, and how this work might fit into some sort of wider context. Thank you also to Egor Shelukhin, Viktor Ginzburg,  Guangbo Xu, Semon Rezchikov, Jae Hee Lee, Zihong Chen, and Todd Liebenschutz-Jones, for historically explaining their work. 

This survey was partially funded by the Max Planck Institute for Mathematics, by the Simons Foundation through award 652299, and by the Heilbronn Institute through a Heilbronn Research Fellowship.}}

\section{Note from the author}
Please do not hesitate to contact the author with idea, corrections or queries. This work is ongoing, and more and more people seem to be interested in ideas in and around this area. 

It should be noted that, while this is a survey paper, there are a few things in here that have not been published elsewhere. In particular, the sketch of mod-$p$ pseudocycles, general equivariant quantum operations, and the additivity problem to name a few. The reasoning is, each of these ideas might form the first part of a larger story, and while the author did not have time to pursue them to completion, they might still be a stepping stone for others to make headway in various directions. Hopefully too they put the content of this survey in greater context. 

Therefore, please feel free to attempt any of the questions/problems that are mentioned in the paper, or to expand on any of the (attempted) definitions. This work is intended not as a new piece of mathematics, but rather a demonstration of what people are working on in the area and more importantly what people {\it could} be working on in the area. 

It should also be made clear that this survey very heavily comes from the author's perspective. I have attempted to elucidate the work of other researchers in the area in Section \ref{sec:recentwork}, and most of these other researchers' works are a lot more geometric and technically impressive than the constructions that I presents in this survey: much of the maths within this text does not go beyond basic differential topology. A survey paper containing all but Section \ref{sec:recentwork} should rather be called ``A differential topological viewpoint of equivariant operations on quantum cohomology", but within the referenced section we see research that touches areas such as representation theory, Hamiltonian dynamics, and Lagrangian Floer theory (among others).

\section{Introduction}
In this survey paper, we aim to discuss some of the ideas in and around the use of equivariance with respect to certain small groups, used to study Gromov-Witten theory. Mainly we consider finite groups, although $S^1$ will appear near the end. We aim to provide a first step towards a general idea of what equivariant symplectic invariants should be, and how they are largely built from topological phenomena. 

\subsection{Background}

Topological methods can form very powerful tools to solve geometric problems: indeed, certain results that seem at first to be difficult can often be reduced to very simple topological problems. In this survey paper we will explore a particular topological method, that of considering equivariant moduli spaces as parametrised by homotopy quotients, as has been used in symplectic geometry (as well as the proceeding applications). 

There is something akin to an equivariant version of the symplectic operadic principle going on here, although the author wants it to be known that no actual, rigorous discussion of operads will be undertaken in this note. More, we mention this in case any interested parties happen to be perusing this note. With this bookkeeping out of the way, we make slightly clearer what we mean by the ``symplectic operadic principle": {\it given an invariant defined by counts of holomorphic curves, the algebraic nature of this invariant is determined by the operadic nature of the underlying space of domains}, discussed e.g. in \cite{abouzaidbottman}. For example, when looking at invariants in quantum cohomology, the natural operad is induced by the topological operad with $n$-ary operations being $\overline{\mathcal{M}}_{0,1+n}$.

Suppose that there exists a group action on some space of domains, by which is invariably meant some manifold consisting of pairs of the form $(\text{pointed nodal Riemann surface}, \text{automorphism})$. In the thesis of Betz \cite{betz}, much work was given to describe how one can combining two such domains. In \cite{cohnor}, this work was couched in a slightly different and more categorical way. In terms of the operadic structure, the idea is this: begin with two Riemann surfaces, with some incoming and outgoing ends/marked points, each Riemann surface with an automorphism group. One can glue the Riemann surfaces, output to input, and the semidirect product of automorphism groups determines an automorphism of the glued surface. This then yields a natural way to combine Riemann surfaces alongside a group action upon them. We will not mention anything else towards operads in this introduction.

This survey's oldest recognisable ancestors are the works of Betz as listed above, and also independently Fukaya (in the paper \cite{fukaya}). Their idea was broadly the following: there is an abstract definition of the Steenrod power operations of \cite{steenrod}, built using the classifying space $B \bZ/p$. However, the classifying space is really a homotopy type, and any construction must be possible with any model of this homotopy type $B \bZ/p$. Therefore, if one defines a specific space with the homotopy type of $B \bZ/p$, arising as the free quotient of a contractible space $E\bZ/p$, then one can define the Steenrod power operations using this specific model. Another, related, example is the cup-$i$ products of \cite{moshtang}. Indeed, Betz and Fukaya determine such a specific model for $B \bZ/p$, in terms of configuration spaces of $p$-tuples of Morse functions. In particular, the space under consideration consists of metric trees, with edges of the tree annotated with Morse functions.

The important point of this is that the construction can be readily extended to quantum cohomology (essentially, given a metric tree, one adds a holomorphic sphere at every vertex of valency at least $3$, see \cite[Section 2]{fukaya}). We will not, here, go into much more depth on these preceding works: every construction in this paper owes their origin at some level to the works of Betz and Fukaya, and understanding the definitions here should be sufficient to understand the constructions of the mentioned authors.

Moving towards the present, a construction of the $i$-th Steenrod power operation is obtained by counting the number of points in moduli spaces parametrised by the $i$-th homology group of some representative of $B \bZ/p$ (which, over $\bF_p$-coefficients, is a $1$-dimensional vector space). So indeed, a slight increase in precision can be obtained for the above constructions. In particular, we can choose a very specific representative of $H_i(B\bZ/p; \bF_p)$ (in particular, one of finite dimension) and define moduli spaces parametrised with this representative (which generally will be e.g. a finite dimensional compact smooth manifold). To give an example, whereas the constructions above for the Steenrod square might involve triples of distinct Morse functions, one can ``get away with" making a single choice of Morse function, but parametrised by $S^{\infty} \ ( `` = E \bZ/p")$, with appropriate regularity/transversality conditions. Note that this idea first arose in places such as \cite{bourgeoisoancea} and \cite{seidel}, and is known as the ``Borel construction".

We will briefly discuss the construction \cite{seidel}, being the most relevant to this survey. In that paper, one considers fixed point Floer homology, and a certain operation on it (which looks like a version of the Steenrod $2$-th power operation in fixed-point Floer theory), and one defines moduli spaces parametrised by a particular homology class of data: specifically, one chooses an injection of $S^{\infty}$ into the space of pairs $(J,H)$, and the associated homology classes of the $\bZ/2$-quotient of this $S^{\infty} \subset \{ (J,H) \}$ is the model for $B \bZ/2$ that parametrises various moduli spaces.

Building on that work, over the course of this paper we will follow the formalism of the more-recent work by the author and Seidel, in \cite{wilkins18} and \cite{covariant}. It is the ideas in these papers that we will try to summarise, albeit in a more general way than was presented therein.

It should be noted here (as it will be later) that various parties have used these Steenrod operations in practise. We will not list them all here, but we give each one (at time of writing) a summary in Section \ref{sec:recentwork}.

\subsection{Structure}

This survey will begin with a general definition of equivariant quantum operations, for monotone symplectic manifolds. Due to the relatively abstract nature of this definition in general, we will then proceed to give examples in the cases of: quantum Steenrod operations, and the quantum Adem relations. We use monotonicity to ensure that our constructions are very concrete, and we do not have to worry about technical details. It should be noted that the definitions as given can be generalised to the semipositive case with some more work, and more generally such definitions should be possible using the work of Bai-Xu \cite{bai2022arnold}, although the author lacked the time and knowledge to achieve this.

We will begin with a general definition of equivariant quantum operations in Section \ref{sec:sec-equivariant-quantum-operations}. This section contains a little preliminary topology, before defining the operations themselves in three steps. 

The next section, Section \ref{sec:sec-properties} will list some general properties of these general operations. This begins with a discussion of additivity, although the technical details will be relegated to the end of the paper in Section \ref{sec:additivity}. We discuss the major methods so far of finding relations between equivariant operations, and discuss how they are quite natural things to study in context. 

The next section, Section \ref{sec:calculations} will be devoted to the particular operations in question studied by the author and others, that being the quantum Steenrod power operations. We discuss first how to (partially) calculate these operations using the covariant constant condition (note that other, much more serious calculations have been achieved by Jae Hee-Lee in \cite{jae1},\cite{jae2}). We then move on to look at the quantum Adem relations. These are relations that have not been greatly explored, although there is certainly interesting content contained therein to do with structural results on quantum Steenrod power operations. 

We then proceed in Section \ref{sec:recentwork} to give an overview of the work of \cite{egorshel1}, \cite{egorshel2}, \cite{semon}, \cite{cgg}, \cite{seidelformal}, \cite{jae1}, \cite{jae2}, \cite{jae3}, \cite{xu}, and \cite{zihong}.

We follow this in Section \ref{sec:sec-floer-invariants} with a discussion of the Floer theoretic invariants, and what can be said about them. 

Finally, in Section \ref{sec:s1-equ}, we consider $S^1$-equivariant quantum operations, what can be said about them and how they relate to this story.

\section{Equivariant quantum operations}
\label{sec:sec-equivariant-quantum-operations}
In this section, we will give the general definition of equivariant quantum operations: although we must note some important caveats. This definition, in its current form, has not (to the author's knowledge) been published previously at this level of generality. However, it follows the method of constructing such operations in previous work of the author and Seidel (e.g. \cite{wilkins18}, \cite{covariant}), which themselves are built on the work of Fukaya \cite{fukaya} and Seidel \cite{seidel}. 

{\bf This is not a straightforward definition, in generality, so readers are advised to do one of the following first:

\begin{itemize}
    \item look at the definitions of quantum Steenrod squares/powers in the citations listed above (or the great definition of quantum Steenrod operations in \cite{jae1}), or
    \item go to the the end of this section (Section \ref{sec:sec-equivariant-quantum-operations}), pick one row from the table thereupon, and then follow along the construction in that specific case, before attempting to swallow the construction in its entirety.
\end{itemize} }

Our constructions will, as previously stated, all assume we are using a monotone symplectic manifold. We reiterate that this is not to say that such constructions would fail for more general symplectic manifolds, but rather that the technical details are irrelevant to the viewpoint of this paper, which stresses the topological underpinnnings. Throughout, we will provide remarks detailing what must be considered carefully for more general symplectic manifolds.

We briefly recall some equivariant topology. Given some group $G$, we let $EG$ be some contractible topological space equipped with a free right action of $G$. We define $BG :=  EG / G$. Given $X$, a topological space equipped with a left $G$-action, we define the homotopy quotient $EG \times_G X$ to be $EG \times X$ quotiented by the $G$-action \begin{equation} \label{equation:convention-for-g-action} g: (v,x) \mapsto (v \cdot g^{-1}, g \cdot x).\end{equation} We note that, unlike in many of the cases that one considers (i.e. commutative groups), this is the only viable choice for the $G$-action on the product. This is a ``nice" approximation of the quotient $G \backslash X$, which may be necessary for example in the case where $X$ is a smooth manifold and the $G$-action is not free. 

We then define the equivariant cohomology to be $H^*_G(X) := H^*(EG \times_G X)$.

\begin{remark}[Conventions]
    We recall \eqref{equation:convention-for-g-action}. In the case where $G$ is Abelian, one can replace the action $$g: (v,x) \mapsto (v \cdot g^{-1}, g \cdot x),$$ with \begin{equation} \label{equation:two-sided-diagonal-action} g: (v,x) \mapsto (g \cdot v, g \cdot x), \end{equation} where we recall that if $G$ is Abelian and acts on $V$, then there is a bijection between left and right actions on $V$, corresponding to $G \lefttorightarrow V \ni g \cdot v := v \cdot g^{-1} \in V \righttoleftarrow G$.

    We also note that, in the case that $G$ is Abelian, there is another obvious choice for the action of $G$ on $EG \times X$, which is \begin{equation}\label{equation:two-sided-diagonal-reversed} g \cdot (v,x) = (g^{-1} \cdot v, g \cdot x),\end{equation} and for most of our constructions we will use this (for example in Section \ref{subsubsec:structure-equiv-ops}.
\end{remark}

\subsection{Preliminary topology}
To start with, we prove something about the topology of classifying spaces. It is possible that the results below have been proven before, but the author has been unable to find a reference.

This subsection can be safely skipped without impacting the understanding of the rest of Section \ref{sec:sec-equivariant-quantum-operations}.

\subsubsection{Filtrations of classifying spaces by finite submanifolds}
\begin{lemma}
\label{lemma:lemma-1}

Given a finite group $G \le \text{Sym}(n)$, there is some choice of the classifying space $EG \rightarrow BG$ such that there is an exhaustive filtration $$E^0 \subset E^1 \subset \dots \subset EG$$ by finite dimensional smooth closed submanifolds, such that $G$ acts smoothly and freely. Defining $B^i = E^i / G$ then yields a filtration of $BG$ by finite dimensional smooth closed submanifolds.
\end{lemma}
\begin{proof}
Note it is sufficient to prove this for $G = \text{Sym}(n)$.

We first observe that the configuration space $\text{Conf}_n(\bR^{\infty})$ is contractible and has a free $G$-action. Fixing some $n$, we denote by $C^i := \text{Conf}_n(\bR^{i})$, observing that $\iota^i: C^i \xhookrightarrow{} C^{i+1}$ using the inclusion $(x_1,\dots,x_i) \mapsto (x_1,\dots,x_i,0)$ on each of the $n$-components. Further, each $C^i$ is an open subset of $\bR^{ni}$ hence is a smooth manifold.

We descend to the following smooth submanifold of $C^i$ immediately: consider the set $D^i$ of $(x_1,\dots,x_n) \in C^i$ such that $|x_i| = 1$ for all $i$. This is certainly a bounded set, and $D^i$ is without boundary, but it is also not closed (e.g. two $x_i$ may become arbitrarily close).
 
 
 Next, we pick some bijection $$\Psi: \Biggr\{ 1, \dots, {{n}\choose{2}} \Biggr\} \rightarrow \{ (s,t) | 1 \le s< t \le n \}.$$ For $i \in \mathbb{Z}_{\ge 0}$, we then define $\alpha^i: D^i \rightarrow (\bR \setminus \{ 0 \})^{{n}\choose{2}}$ in the following way. Suppose that $1 \le k \le {{n}\choose{2}}$ with $\Psi(k) = (a,b)$. Then the  $k$-th component of $\alpha^i(x_1,\dots,x_n)$ is $|x_a - x_b|$. 
 
 Then for sufficiently large $i$, there is a choice of regular $\delta \in \bR \setminus \{ 0 \}$ sufficiently small such that $E^i := (\alpha^i)^{-1}((\delta,\dots,\delta)) \subset C^i$ is nonempty. 
 
 For nonemptiness, observe first that $x_1$ may be chosen freely, that $x_2$ must be chosen in a sphere of radius $\delta$ around $x_2$, that $x_3$ must be chosen in the intersection of two spheres of radius $\delta$, and so forth. If $i$ is sufficiently large, i.e. $i>n$ and $\delta < 1$, then there always exist a collection of such points. For regularity, as the set of regular values is of Baire second category, for each coordinate of $\alpha^i$ there is some Baire second category set of regular $\delta \in [0,1]$, and the intersection over all coordinates of $\alpha^i$ is also a Baire second category set. In particular, this intersection is nonempty and there some $\delta$ such that $(\delta,\dots,\delta)$ is regular for $\alpha^i$. By the preimage theorem, each $E^i$ is a finite-dimensional smooth manifold, and each has a free $G$-action inherited from $C^i$. Further, $\iota^i$ induces an inclusion $E^i \subset E^{i+1}$.

As $E^i \subset D^i$ it remains bounded, but being $(\alpha^i)^{-1}(\delta,\dots,\delta)$ it is also closed. Hence, it is compact. As each $E^i$ inherits a free $\text{Sym}(n)$ action, it thus remains to prove that $\tilde{EG} := \bigcup_i E^i$ is contractible (and hence $\tilde{EG}$ has the required properties of the universal cover of a classifying space, with the $E^i$ being the exhausting submanifolds with the sought-after properties). Note that because $\tilde{EG}$ is exhausted by smooth manifolds, then $\tilde{EG}$ is a CW-complex, so it suffices to prove that all of the homotopy groups of $\tilde{EG}$ vanish.  

We will now reintroduce $n$ to our notation, so $E^i = E^i(n)$, dependent on the choice of $n \in \bZ_{\ge 0}$. Note also that there is a smooth fibre bundle $E^i(n) \rightarrow E^i(n-1)$, induced by the projection $$(x_1,\dots,x_n) \rightarrow (x_1,\dots,x_{n-1}).$$ The fibre is the intersection of $n$ different (transversely and non-emptily intersecting) $i-1$-spheres. In particular, if $i$ is sufficiently large then the fibre of $E^i(n) \rightarrow E^i(n-1)$ is $(i-n-1)$-connected. We also observe that $E^i(1)$ is an $i-1$-sphere, hence is $(i-1)$-connected. Iteratively, using the long-exact-sequence of a fibration, we thus demonstrate that $E^i(n)$ is $(i-n-1)$-connected. Hence, fixing any choice of $n$, we observe that $\tilde{EG}$ must have vanishing homotopy group $\pi_n$, thus is contractible. 
\end{proof}

It should be noted that this immediately generalises to yield a filtration of $EG \times_G X$ by finite dimensional smooth compact submanifolds, for any smooth compact manifold $X$. More generally, if $X$ is a finite cell complex then we can prove identically the following:

\begin{corollary}
\label{corollary:corollary-to-lemma-1}
Given a finite group $G \le \text{Sym}(n)$, and $EG$ as in Lemma \ref{lemma:lemma-1}, and some finite cell complex $X$, then $B^i = E^i \times_G X$ yields a filtration of $E \times_G X$ by finite cell complexes.
\end{corollary}

\begin{example}
\label{example:z-mod-p}
In the case where $G = \bZ/p$, we may pick $EG = S^{\infty} \subset \bC^{\infty}$. Elements are of the form $w=(w_0,w_1,\dots,w_i,0,\dots)$ such that each $w_j \in \bC$, and $\sum_j |w_j|=1$. Then $E^i = S^{2i+1}$ consists of elements of $EG$ such that $w_j = 0$ for $j > i$. We note that $E^i$ has cells $$\Delta_{2i} = \{ w \in S^\infty \;:\; w_{i} \geq 0, \; w_{i+1} = w_{i+2} = \cdots = 0\}, $$
$$\Delta_{2i+1} = \{ w \in S^\infty \;:\; e^{-i \theta} w_i \geq 0 \text{ for some $\theta \in [0,2\pi/p]$} \}.$$ These cells, in addition to their $\bZ/p$-orbits, yield a cellular decomposition of $E$.
\end{example}

If $G$ contains no permutations of negative sign, then the action of $G$ on the $E^i$ is automatically orientation preserving: to see this, locally $E^i$ is a product of copies of $\bR$, and the action of $G$ is the permutation action. An even permutation automatically preserves the orientation on a product of copies of $\bR$. In particular, $B^i$ are automatically oriented (inheriting an orientation from $E^i$). However, if this is not the case then we must go a step further. If $G$ contains some permutation of negative sign, and if $G$ is also commutative then we can replace $E^i$ with $E^i \times E^i$ equipped with the diagonal action. Then $G$ acts orientation preservingly, and the resulting filtration $B^{2i} := E^i \times_G E^i$ is orientable. Notice then that the filtration only involves even dimensional smooth closed manifolds. In the case when $G$ is non-commutative and has some negative signs, we cannot ensure that $BG$ is stratified by orientable smooth closed manifolds. 


As is standard, there is thus a smooth triangulation into $i$-simplices for each $E^i$ (or $E^i \times E^i$), which is coherent (by which we mean it respects $E^i \subset E^{i+1}$) and is $G$-invariant. 

\begin{corollary}
\label{cor:corollary-1}
For a finite $n$-cell complex $X$ with a $G$-action, we may construct $H_*^G(X)$ as the union of the cellular homology $$\bigcup_i H_*(E^i \times_G X),$$ with respect to some choice of coherent smooth triangulation of the $B^i$. In particular, every $A \in H_j^G(X; \bZ)$ may be represented by some finite union of $j$-cells (perhaps with repetition).
\end{corollary}
\begin{proof}
Using Corollary \ref{corollary:corollary-to-lemma-1}, we observe that any $A \in H_j^G(X)$ must arise from some $A^i \in H_j(E^i \times_G X)$. As $E^i$ is a compact smooth manifold and $X$ is a finite cell complex, we know that the homotopy quotient $E^i \times_G X$ has a smooth finite triangulation by $n \dim(B^i)$-simplices, and thus any element of $C_j(E^i \times_G X)$ is written as a sum of $j$-cells.
\end{proof}

\subsubsection{Splitting of operations over orbits}
\label{subsubsec:splitting-orbits}
We will fix some notation for later use: we define $\{ O_1,\dots, O_I \}$ to be the set of orbits of the action $G \lefttorightarrow \{1,\dots,n \}$. Let $\text{Fix}_j$ be the set of $g \in G$ that fix every element of $O_j$ (alternatively, the intersection of all of the stabilisers of the elements of $O_j$). Observe that each $\text{Fix}_j$ is normal in $G$. Let $G_j$ be the set of elements of $G$ that act trivially on $$O_1 \sqcup \dots \sqcup O_{j-1} \sqcup O_{j+1} \sqcup \dots \sqcup O_I,$$ i.e. $G_j = \cap_{i \neq j} \text{Fix}_i$. Suppose that each of the $G_j$ pairwise commute with each other. It is immediate that each $G_j$ is a subgroup, and indeed that $G \cong G_1 \times \dots \times G_I$. Hence, $$BG = BG_1 \times \dots \times BG_I.$$ In particular, by restricting to each $O_j$ equipped with the action of $G_j$, we can treat each orbit of the permutation space independently while defining the operations in Corollary \ref{corollary:well-defined-operations}.

\subsubsection{Cohomology of the classifying space acts on equivariant cohomology}

We recall, as a final piece of preliminary topology, that given some cell complex $X$ equipped with a $G$ action. Then there is a map $H^*(BG) \rightarrow H^*_G(X)$ that is induced by projection $EG \times_G X \rightarrow BG$. This then provides us with a pairing \begin{equation} \label{eq:bg-pairing} *:H^*(BG) \otimes H^*_G(X) \rightarrow H^*_G(X).\end{equation}

\subsubsection{Equivariant cohomology}
\label{subsubsec:equiv-cohom}

We will give some explicit constructions of the $G$-equivariant cohomology of a smooth manifold $X$, for the readers' convenience.

We know that in general the definition is $H^*_G(X) := H^*(EG \times_G X)$. However, in the case where $G = \mathbb{Z}/p$, we can be more explicit. Let us assume that $p>3$ (the $p=2$ case is easier, adding the relation $e^2 = u$). Pick $g:X \rightarrow X$ the action of $1 \in \mathbb{Z}/p$. 

Another way to calculate $H^*_{\mathbb{Z}/p}(X)$ is as follows: begin with $$C^*_{\mathbb{Z}/p}(X) = C^*(X)[u] \otimes \Lambda(e).$$ Then define the differential as follows:

\begin{equation} \label{equation:equivariant-differential} \begin{cases} \begin{array}{ll} d_{\mathbb{Z}/p}(x u^j) = & dx u^j + (-1)^{|x|} (g - \text{id}) \cdot x u^j e   \\ d_{\mathbb{Z}/p}(x u^j e) = & dx u^j e+ (-1)^{|x|} (\text{id} + g + g^2 + \dots + g^{p-1}) \cdot x u^{j+1}.   \end{array} \end{cases} \end{equation}

One notices too that we may choose $S^{\infty} \subset C^{\infty}$ to be the model for $E \mathbb{Z}/p$, with the induced action of multiplication by $e^{2 \pi i / p}$. Then there is a cellular decomposition of $S^{\infty}$, with for each  $k=0,1,\dots$ some collection of discs $\Delta^k_{0},\dots, \Delta^k_{p-1} \subset S^{\infty}$ of dimension $k$, such that $e^{2 \pi i / p} \cdot \Delta^k_j = \Delta^k_{j+1}$ for all $j$ (see e.g. \cite[Section 2]{covariant}). Recalling Example \ref{example:z-mod-p}, if $k$ is even then these are smooth discs, and if $k$ is odd then these are discs that are smooth away from codimension $2$ corners. Then $$d \Delta^{2k}_0 = \Delta^{2k-1}_0 + \dots + \Delta^{2k-1}_{p-1},$$ and  $$d \Delta^{2k+1}_0 = \Delta^{2k}_1 + \dots + \Delta^{2k}_{0}.$$ One can see that this is dual to the situation in \eqref{equation:equivariant-differential}. In particular, one can use the above discs to calculate $H_*^{\mathbb{Z}/p}(X)$ via $C_*^{\text{cell}}(E \mathbb{Z}/p) \otimes C_*(X)$.

\subsection{Equivariant quantum operations}
\label{sec:eq-qu-op}

Firstly, we would like to recall the definition of pseudocycles and bordisms between them. We will only consider a compact manifold $M$ (for brevity). 

\begin{definition}
A smooth map $f: X \rightarrow M$ is called a {\it pseudocycle of dimension $x$} if $X$ is an oriented manifold of dimension $x$, and $$\Omega_f := \bigcap_{\begin{array}{c} K \subset X \\ K \text{ compact} \end{array}} \overline{f(X \setminus K)}$$ is covered by a union of smooth maps $g_i: X'_i \rightarrow M$ where $\dim X'_i \le x-2$.  
\end{definition}

\begin{definition}
Given two pseudocycles $f_0: X_0 \rightarrow M, \ f_1: X_1 \rightarrow M$ of dimension $x$, a smooth map $H: Y \rightarrow N$ is called a {\it pseudocycle bordism between $f_0$ and $f_1$} if $Y$ is a smooth manifold with boundary, $Y$ is of dimension $x+1$, and the boundary $\partial Y = X_1 - X_0$ (the negative sign denoting negative orientation) such that $H|_{X_i} = f_i$. Further, $\Omega_H$ has dimension at most $x-1$.
\end{definition}

It is a result of Zinger \cite{zinger} that there is an isomorphism between $H_*(M)$ and the free Abelian group generated by pseudocycles modulo pseudocycle bordisms (i.e. modulo relations of the form $[f] - [g]$ for pseudocycles $f$ and $g$ any time there is a pseudocycle bordism between $f$ and $g$). If we were considering noncompact manifolds, then we would either have to strengthen our definition (i.e. ``precompact") or consider ``locally finite pseudocycles", as in \cite{lf-pseudocycles}, to obtain instead locally-finite homology.

To keep everything explicit, we will assume that $M$ is a closed monotone symplectic manifold (although a similar, albeit more complicated, definition should exist for convex and/or weakly monotone symplectic manifolds). We will fix a compatible almost complex structure $J$. This yields a well defined Chern class for $TM$, so we say $c_1(M) := c_1(TM)$.

Regarding quantum cohomology, the reader will be assumed to be somewhat familiar: however, references are provided for the full details, and we recall the outline. Note first that $$QC^*(M) := C^*(M;\Lambda),$$ where $$\Lambda:= \bK [ q ]$$ is a Novikov ring over some formal variable with $q$ of index $2$. The differential is just the usual cohomological differential, extended linearly over $\Lambda$. The cohomology with respect to this differential is $QH^*(M) = H^*(M; \Lambda)$.

There is a so-called quantum product on $QH^*(M)$, which we give the broad definition below (for technical details, recall \cite{jhols}). Preliminarily, given $x,y \in H^*(M)$, we apply Poincar\'e duality to obtain corresponding homology classes and then under the isomorphism in \cite{zinger} we respectively for $x,y$ obtain pseudocycles $f:X \rightarrow M$ and $g: Y \rightarrow M$. 

\begin{remark}
To sidestep internalising the notion of the omega-limit set $\Omega_f$ of a pseudocycle $f$, one can assume that the Poincar\'e dual of $x,y$ are represented by embedded submanifolds if it is helpful. 
\end{remark}

Further, let $\mathcal{M}_k(J)$ consist of $J$-holomorphic maps $u:S^2 \rightarrow M$ such that $c_1(M)(u_*[S^2]) = k$ and define $ev_{k,i}$ to be evaluation at $i \in S^2 \cong \mathbb{C} \cup \{ \infty \}.$ Then, for generic choices of $f,g$, we define the $k$-th quantum product $$PD(x *_k y) := \left( ev_{k,3}: ev_{k,1}^{-1}(\text{im}(f)) \cap ev_{k,2}^{-1}(\text{im}(g)) \rightarrow M \right),$$ as a pseudocycle, and the total (small) quantum product $$H^*(M) \otimes H^*(M) \mapsto QH^*(M), \quad (x,y) \mapsto \sum_{j \ge 0} x *_j y q^j,$$ extending bilinearly over $\Lambda$ in both inputs. To see a more rigorous description of the quantum product, see \cite{jhols}. Note that we will interchangeably use $PD: H^*(M) \rightarrow H_{\text{dim} M -*}(M)$ and $PD: H_*(M) \rightarrow H^{\text{dim} M - *}(M)$, as there is no ambiguity.

\subsubsection{Setup and structure}
\label{subsubsec:structure-equiv-ops}

Suppose now that $n \in \mathbb{Z}_{\ge 1}$. Suppose that $G \le \text{Sym}(n)$ and $p | |G|$.


Throughout this section, we will fix these $n,G,p$. All (co)homology henceforth will be taken with $\bF_p$ coefficients unless otherwise stated. We recall that $\overline{\mathcal{M}}_{0,1+n}$ is defined as follows, following \cite{jholssympl}.

A stable nodal genus $0$ curve is a collection of finitely many spheres $S^2$, where some pairs of two spheres are attached at a node, so that if one forms a graph where each sphere is a vertex and each node an edge, this graph is a connected tree.

An $1+n$-pointed stable nodal genus $0$ curve is a stable nodal genus $0$ curve with $1+n$ marked points distinct from the nodes, so that each sphere has $\# \text{ marked points} + \# \text{ nodes } \ge 3$. Generally, the first marked point is distinguished, and we call the sphere on which this marked point is found the ``principal component" of the pointed stable nodal genus $0$ curve. 

\begin{definition}
For $n \ge 3$, $\overline{\mathcal{M}}_{0, 1+ n}$ is the space of stable curves of genus zero with $1+n$ marked points, up to automorphisms (i.e. M\"obius transformations concurrently on each sphere).
\end{definition}

While this is not necessarily the most currently used version of this definition, it will suffice for our purposes (in particular, we only really care about $n \ge 2$). 

An alternative viewpoint is that $\mathcal{M}_{0,1+n}$ consists of $1+n$ distinct marked points $z_0,z_1,\dots, z_n$ on the sphere, up to reparametrisation by $PSL(2,\mathbb{C})$, with a compactification $\overline{\mathcal{M}}_{0,1+n}$ consisting of: if one or more marked points collide, then they bubble off into a tree of spheres, with the arrangements of the marked points on the bubbles being dictated by the relative speed and angle with which the points approach each other. 

With this, we notice that $\overline{\mathcal{M}}_{0,1+2}$ is a point, $\mathcal{M}_{0,1+3} \cong S^2 \setminus \{0,1,\infty\}$ by triple-transitivity hence $\overline{\mathcal{M}}_{0,1+3} \cong S^2$, and we can obtain larger $\overline{\mathcal{M}}_{0,1+n}$ by taking blow-ups of products of $S^2$.

We note that there is an action of $G$ on  $\mathcal{M}_{0,1+n}$ by permuting the last $n$ marked points, and that this naturally extends to $\overline{\mathcal{M}}_{0,1+n}$, the compactification.

Given some closed submanifold $A \subset \mathcal{M}_{0,1+n}$ acted on by $G$, we demonstrate how to construct an operation:

$$Q: H_*(EG \times_G A) \otimes QH^*(M)^{\otimes I} \rightarrow QH^*(M),$$ recalling that $I$ is the number of $G-$orbits of $\{ 1 , \dots, n \}$. We note that $EG \times A$ is exhausted by smooth closed manifolds ($E^i \times A$), and hence arguing as in Corollary \ref{cor:corollary-1}, that every element of $H_i(EG \times_G A)$ may be represented as some finite union of smooth $i$-cells. The $G$-action is $(gv, m) \sim (v,gm)$ for $g \in G$. We let $\rho : EG \rightarrow BG$ be the projection.

We thence do the following steps: 

\begin{enumerate}[label=(\Alph*)]
    \item suppose that $C$ is some finite union of $i$-cells in $EG \times_G A$ representing an element $[C] \in H_i(EG \times_G A)$ (for some $i$). Suppose also that $j \in \mathbb{Z}_{\ge 0}$. We will thus define an operation $$Q_j(C) : C^*(M)^{\otimes n} \rightarrow C^*(M),$$
    \item demonstrate that this descends to a well-defined map on cohomology, $$\begin{array}{c} \overline{Q_j}(C) : \bigotimes_{i=1}^I H^*(M) \rightarrow H^*(M) \\ \overline{Q_j}(C)(x_1 \otimes \dots \otimes x_I) = Q_j(C)(x_1^{|O_1|} \otimes \dots \otimes x_I^{|O_I|}),\end{array}$$
    \item demonstrate that $\overline{Q_j}(C)$ only depends on the homology class of $C$.
\end{enumerate}

If all of these hold, then we can define $$Q([C] \otimes x_1 \otimes \dots \otimes x_I) := \sum_{j \ge 0} \overline{Q_j}(C)(x_1 \otimes \dots \otimes x_I) q^j.$$

It is important here, as we will discuss later, to notice that \'a priori there is no reason to assume that this map is additive.

In order to define $Q([C]  \otimes x_1 \otimes \dots \otimes x_I)$, we will have to make a choice. First, recall that $G$ acts on $\{1,\dots,n\}$, with orbits $\{O_1,\dots,O_I \}$. For each $i$, let $S_i \subset G$ be the stabiliser of the smallest element of $O_i$. Recall that we have fixed some Morse function $f:M \rightarrow \bR$. We wish to define perturbations $f^i_{v,m,s}: M \rightarrow \bR$ for $i=1,\dots,I$, and $m \in A$, and $v \in EG$, and $s \in [0,\infty)$ such that:
\begin{enumerate}
    \item $f^i_{v,m,s} = f$ if $s \ge 1$,
    \item $f^i_{gv,gm,s} = f^i_{v,m,s}$ for $g \in G$,
    \item $f^i_{gv,m,s} = f^i_{v,m,s}$ for $g \in S_i$.
\end{enumerate}

If we make a generic choice, then the moduli spaces we define below will be well-defined smooth manifolds. We assume without loss of generality (i.e. up to isomorphism of groups) that $O_1 = \{1,\dots, |O_1| \}, \ O_2 = \{|O_1|+1,\dots, |O_1| + |O_2| \}, \dots, O_I = \{ n - |O_I|+1,\dots,n\}$. For each $i=1,\dots,I$ and $j=1,\dots,|O_i|$, we denote by $O_i(j)$ the $j$-th smallest element of $O_i$. Then for all $i,j$ we choose $g_{i,j} \in G$ to be such that $O_i(1) = g_{i,j} O_i(j)$.

\subsubsection{Step A}
\label{sec:stepA}
Recall that we are modelling homology via cellular chains. Indeed, we can do this with cells such that they are smooth manifolds up to codimension $2$, and removing these codimension $2$ pieces we obtain smooth manifolds with boundary. Suppose that $B \in C_i(EG \times_G A)$ is some $i$-cell. We will first define an operation $Q_j(C) : C^*(M)^{\otimes n} \rightarrow C^*(M).$ To avoid having to consider manifolds with corners, we denote by $B^{wc}$ to be $B$ less its codimension $2$ substrata (hence a manifold with boundary).

For $m = (z_0,z_1,\dots,z_n) \in A$, define $S_m$ to be $S^2$ with $(-\infty,0]$ attached at $z_0$ and a copy of $[0,\infty)$ attached at each of $z_1,\dots,z_n$ (we call these copies respectively $[0,\infty)_1,\dots, [0,\infty)_n$). We will also pick a lift $\tilde{B}$ of $B$ to a union of $i$-cells in $EG \times A$.

For $x_0,\dots,x_n \in \text{crit}(f)$, define the moduli space $\mathcal{M}_{j}(B;x_0,x_1,\dots,x_n)$ to consist of tuples $(v,m, u)$ such that:
\begin{enumerate}
   \item $(v, m) \in \tilde{B}^{wc} \subset EG \times A$,
   \item $u: S_m \rightarrow M$,
  \item $u$ is $J$-holomorphic on $S^2$ of Chern number $j$,
 \item $u|_{(-\infty,0]}$ is a $-\nabla f$-flowline, with $u(-\infty) = x_0$,
 \item  $u|_{[0,\infty)_{O_i(j)}}$ is a $-\nabla f^i_{m,g_{i,j}v,s}$-flowline. 
 \item $\lim_{s_k \in [0,\infty)_k, \ s_k \rightarrow \infty} u(s_k) = x_k$.
\end{enumerate}

In the first property, by ``$(v,m) \in \tilde{B}^{wc}$" we mean that $B = \sum_k B_k$ is a sum of cells, which we can lift to $EG \times A$, and we consider triples $(i,v,m)$ such that $(v,m) \in \tilde{B_i}^{wc}$. See Figure \ref{fig:equivariantoperations}

\begin{figure}
    \centering
\includegraphics[scale=0.6]{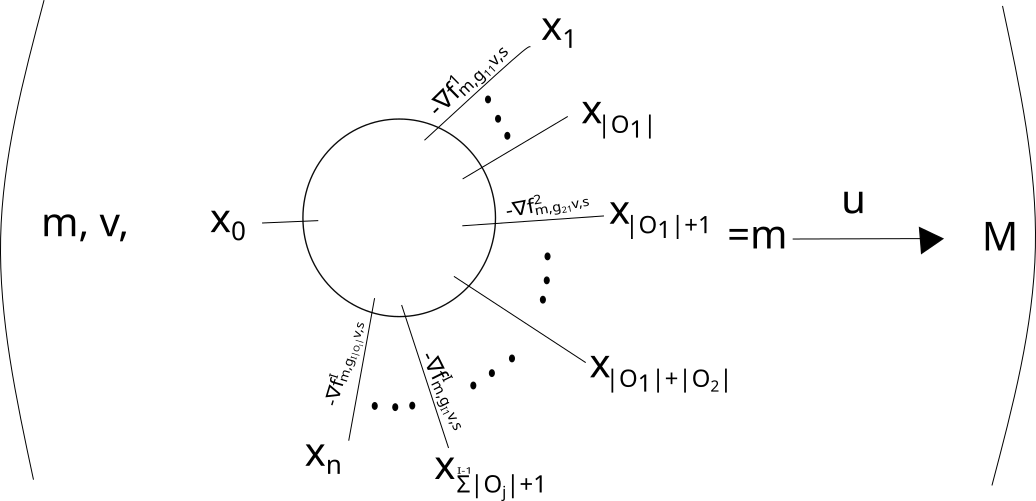}
    \caption{The elements of the moduli space. We have adorned the domain with the conditions that we need $u$ to satisfy on (most) of its bits.}
    \label{fig:equivariantoperations}
\end{figure}

Importantly, condition $(1)$ is independent of the choice of lift $\tilde{B}$: given some other choice of $i$-cell over $B$, this will be of the form $g \tilde{B}$; then $(v,m,u) \mapsto (gv, gm, u)$ yields a bijection of moduli spaces. Further, condition $(5)$ is well-defined independent of the choices of $g_{i,j}$ because of the assumption on $f^i_{v,m,s}$ being invariant under multiplication by $S_i$ on the $v$-coordinate. In particular, if $g_{i,j}$ and $g'_{i,j}$ are two choices, then $g'_{i,j} \circ g^{-1}_{i,j} \in S_i$. Thus, $g'_{i,j} = s_{i,j} g_{i,j}$ for some $S_{i,j} \in S_i$, and so $$f^i_{m,g'_{i,j}v,s} = f^i_{m,s_{i,j} g_{i,j}v,s} = f^i_{m,g_{i,j}v,s}.$$ 

Now one needs to consider boundaries of this moduli space. There are two possibilities, corresponding either to $\partial B^{wc}$ or to a single Morse breaking on any one of the incoming and outgoing flowlines. We will briefly analyse the former of these two cases. 

Suppose that $$(v,m,u) \in \partial \mathcal{M}_{j}(B;x_0,x_1,\dots,x_n)$$ sits over $(v,m) \in \partial B^{wc}$, and suppose that $g \in G$ is such that  $(gv, g^{-1}m) \in \partial B^{wc}$. Then $$(gv, g^{-1}m,u) \in  \partial \mathcal{M}_{j}(B;x_0,x_{g1},\dots,x_{gn}),$$ where all of the moduli space conditions $(1)-(6)$ are satisfied, and all apart from $(5)$ are obvious. To demonstrate that $(5)$ holds, if $g(O_i(j)) = O_i(k)$ then $g g^{-1}_{i,j} O_i(1) = g^{-1}_{i,k} O_i(1)$. So $g_{i,k} g g^{-1}_{i,j} \in S_i$. In particular, $f^i_{gm,g_{i,k} g v,s} = f^i_{gm,g_{i,j} v,s}$. Hence $g$ acts by relabelling the elements of $O_i$ for each $i$. In particular, we can confirm that $(gv, g^{-1}m,u)$ is within $\partial \mathcal{M}_{j}(B;x_0,x_{g1},\dots,x_{gn})$. Further, if for every $k=1,\dots,I$ and every $i,j \in O_k$ we have $x_i = x_j$, and if additionally $dB = 0$ (hence the number of $g$ satisfying $(gv, g^{-1}m) \in \partial B^{wc}$ is divisible by $p$), then we notice that solutions $(v,m,u)$ with $(v,m) \in \partial B^{wc}$ come in families of size $0$ modulo $p$. Hence, in characteristic $p$, the only boundary points come from Morse breaking. This is a first piece of evidence that closed $B \in C_*^G(A)$ will yield well-defined homology operations, which we formalise in Section \ref{sec:stepB}. 

This moduli space $\mathcal{M}_{j}(B;x_0,x_1,\dots,x_n)$ is a union of smooth manifolds-with-boundary of dimension \begin{equation}
    \label{eq:moduli-dimension} |x_0| - \sum_{k=1}^n |x_k| + 2j + i. 
\end{equation}

Observe that this moduli space can be compactified by adding a set covered by a union of manifolds of dimension at most $|x_0| - \sum_{k=1}^n |x_k| + 2j + i-2$. This is a standard monotone sphere bubbling argument.

\begin{definition}
\label{definition:moduli-def}
For each $i$-cell $B$ in our cellular decomposition of $EG \times_G A$, define $$\begin{array}{l} Q_j(B): C^*(M)^{\otimes n} \rightarrow C^*(M) \\ Q_j(B)(x_1 \otimes \dots \otimes x_n) = \sum_{x_0 \in \text{crit}(f) : \eqref{eq:moduli-dimension}  \text{ vanishes}} \#_p \mathcal{M}_{j}(B;x_0,x_1,\dots,x_n) \cdot x_0, \end{array}$$ where we observe that $\mathcal{M}_{j}(B;x_0,x_1,\dots,x_n) $ is a smooth $0$-dimensional manifold whenever \eqref{eq:moduli-dimension} vanishes, and $\#_p$ is the count of points mod-$p$.

We then extend each $Q_j(B)$ linearly, and further define $$Q_j(B + B') := Q_j(B) + Q_j(B').$$
\end{definition}

\subsubsection{Step B}
\label{sec:stepB}
We first note that we may without loss of generality (using K\"unneth) assume that the cellular decomposition of $EG \times_G A$ is coherent, under the projection to $BG$, with respect to some cellular decomposition of $BG$. We also can consider each cell as a cell in $EG \times A$ by the lifting property, thus associate $C_*(EG \times_G A) = C_*(EG) \otimes_G C_*(A)$. This yields us a ``nice" basis for  $C_*(EG \times_G A) = C_*(EG) \otimes_G C_*(A) = (C_*(EG) \otimes C_*(A))/G$.


Before the theorem, we observe that there is a chain map $$\cap: C_*(EG \times_G A) \times C^*(BG) \rightarrow C_*(EG \times_G A),$$ via composing the cap product with the cohomology map induced by projection $C^*(BG) \rightarrow C^*(EG \times_G A)$.

Further, we observe that given any vector space $A$ with $G$-action, and any field $F$, there is an injection $$\text{Hom}(A/G, F) \xhookrightarrow{} \text{Hom}(A,F).$$ In fact, this lands in the $G$-invariant homomorphisms. In particular, any element of $C^*(EG \times_G M^{\times n}; \mathbb{F}_p)$ may be considered as a $G$-invariant element of $ C^*(EG) \otimes C^*(M^{\times n})$. 

We note that, there exists a choice of cellular decomposition on $BG$ such that the lifted cellular decomposition on $EG$ consists of elements of $C_*(EG)$ in families of size $|G|$, i.e. each family is associated with an element of $C_*(BG)$ under the quotient map. In particular, in such a standard cellular decomposition of $(EG, BG)$, if $\pi: EG \rightarrow BG$ is projection then each nonzero element of $C_*^{\text{cell}}(BG)$ has exactly $|G|$-preimages. We obtain, using the dual basis associated to the basis of cells, a $|G|$ to $1$ map $$\Psi: C^*_{\text{cell}}(EG) \rightarrow C^*_{\text{cell}}(BG),$$ as follows: just as the cells form a basis of the vector space of cellular chains, the dual cocells form a dual basis of $C^*(EG)$ and $C^*(BG)$, and if $D$ is a cell in $C_*(EG)$ and $D^*$ its dual in $C^*(BG)$, and $[D]$ its image in $C_*(BG)$, then $\Psi(D^*) := [D]^*$. This yields an additive chain map (this obviously commutes with the boundary map, although the resulting homology map vanishes hence is uninteresting).

Further, we will fix in $EG$ a choice ``$1$" $\in C^0(EG)$, as the representative of $1 \in H^0(EG) \cong \mathbb{F}_p$.

\begin{theorem}
\label{theorem:main-theorem}
Given $C \in C_*(EG \times_G A)$ (a sum of $i$-cells), if $dC = 0$ then $Q_j(C)$ extends to a chain map $Q^{ex}_j(C): C^*(EG \times_G M^{\times n}) \rightarrow C^*(M)$. Using the discussion above this theorem, i.e. writing an element of $C^*(EG \times_G M^{\times n})$ as an element of $C^*(EG) \otimes C^*(M^{\times n})$, then Definition \ref{definition:moduli-def} is on the chain-level $Q_j(C)(-) := Q^{ex}_j(C)(1 \otimes -)$, and the following {\it consistency equation} may be chosen to hold for any $\tilde{b}$: \begin{equation}\label{equation:strapping} Q^{ex}_j(C)(\tilde{b} \otimes \underline{x}) := Q^{ex}_j(C \cap \Psi(\tilde{b}))(1 \otimes \underline{x}).\end{equation} 
\end{theorem}
\begin{proof}
To begin with, we note that Equation \eqref{equation:strapping} is always well-defined, as $\Psi$ is an additive homomorphism.

We now need to prove that the stated map is a chain map. In particular, we want to show that \begin{equation} \label{equation:chain-map} d Q_j(C)(\tilde{b} \otimes \underline{x}) =(-1)^{|\tilde{b}|} Q_j(C)(\tilde{b} \otimes d \underline{x}) + Q_j(C)(d \tilde{b} \otimes \underline{x}). \end{equation} Using \eqref{equation:strapping}, this amounts to prove that $$d Q_j(C \cap \Psi(\tilde{b}))(1 \otimes \underline{x}) =(-1)^{|\tilde{b}|} Q_j(C \cap \Psi(\tilde{b}))(1 \otimes d \underline{x}) +   Q_j(C \cap \Psi(d \tilde{b}))(1 \otimes \underline{x}).$$ Using the fact that $dC=0$, observe that $d( C \cap \Psi(\tilde{b})) = C \cap d \Psi( \tilde{b})$. Recall from above the theorem that $\Psi$ commutes with the differential. Hence, we must in fact prove that \begin{equation}\label{eq:strapping2} Q_j(C \cap \Psi(\tilde{b}))(1 \otimes \underline{x}) = (-1)^{|\tilde{b}|}  Q_j(C \cap \Psi(\tilde{b}))(1 \otimes d \underline{x}) +  Q_j(d(C \cap \Psi( \tilde{b})))(1 \otimes \underline{x}).\end{equation}

Returning to Definition \ref{definition:moduli-def}, if we consider the endpoints of a $1$-dimensional version of the moduli space from the referenced definition, then the three sorts of boundaries that may occur are from the boundary of and $i$-cell (i.e. $dE$), or breaking of on of the incoming or the outgoing Morse flowlines. So in fact, $d \circ Q_j(E) = Q_j(E) \circ d + Q_j(dE)$, which proves \eqref{eq:strapping2}.
\end{proof}

The following argument uses a proof due to Seidel (see \cite[Lemma 2.5]{covariant}) to greatly simplify it. We recall that $G \subset \text{Sym}(n)$ as a permutation group, the set $\{1,\dots,n \}$ splits into $G-$orbits $O_1,\dots,O_I$.
\begin{lemma}
The map $$\begin{array}{c} \text{pow}_{G} : C^*(M^I) \rightarrow (C^*(EG) \otimes C^*(M^{|G|}))^G \xleftarrow{\cong} C^*(EG \times_G M^{|G|}), \\  x_1 \otimes \dots \otimes x_I \mapsto 1 \otimes x_1^{\otimes |O_1|} \otimes \dots \otimes x_I^{\otimes |O_I|},\end{array}$$ is well-defined on cohomology, where we have (abusing notation) chosen $1 \in C^0(EG)$ to be the image under $\pi^*$ of some representative of $1 \in H^0(BG)$.
\end{lemma}
\begin{proof}
We first note that because of our choice of $1 \in C^0(EG)$, the map as given is well-defined (i.e. it lands intermediately in $G$-invariant chains). Further, the second arrow is an isomorphism of chain complexes. That it is a chain map is a combination of K\"unneth and the fact $\pi^*$ is a chain map. For surjectivity, any such $G$-equivariant homomorphism immediately descends to a homomorphism in the quotient, and for injectivity notice that any homomorphism that lifts to a $G$-invariant homomorphism that is $0$ must be $0$ on the quotient.

If $x$ and $x'$ both represent the same thing in $H^*(M)$, then for any fixed $x_2, \dots\ x_I$ we obtain that $$1 \otimes x^{\otimes |O_1|} \otimes x_2^{\otimes |O_2|} \otimes \dots \otimes x_I^{\otimes |O_I|}$$ and $$1 \otimes (x')^{\otimes |O_1|} \otimes x_2^{\otimes |O_2|} \otimes \dots \otimes x_I^{\otimes |O_I|}$$ both represent the same thing in $H^*(EG \times_G M^{\times |G|})$: firstly, $x-x' = d z$, so use inclusion of complex $0 \rightarrow z \rightarrow \langle x,x' \rangle \rightarrow 0$ and map to complex $0 \rightarrow 0 \rightarrow \langle y \rangle \rightarrow 0$. The same is true in any ``coordinate", and using K\"unneth we may treat each ``coordinate" separately.
\end{proof}

\begin{corollary}
\label{corollary:well-defined-operations}
The map $$\begin{array}{c} \overline{Q_j}(C) : \bigotimes_{i=1}^I H^*(M) \rightarrow H^*(M) \\ \overline{Q_j}(C)(x_1 \otimes \dots \otimes x_I) = Q_j(C)(\text{pow}_G(x_1 \otimes \dots \otimes x_I)),\end{array}$$ is well-defined.
\end{corollary}

\begin{remark}
Observe that $\text{pow}_{O_i}$ is not additive. Thus, the operations $\overline{Q_j}(C)$ need not be additive.
\end{remark}

\subsubsection{Step C}
\begin{theorem}
\label{theorem:bordisms}
The operation $\overline{Q_j}(C)$ only depends on the homology class of $C$.
\end{theorem}
\begin{proof}
Given homologous cellular chains $C, C'$, there is some cellular chain $\partial D = C - C'$ mod-$p$. Each cell in the chain $D$ yields a $1$-dimensional moduli space (a $1$-dimensional manifold with boundary), and we then construct a $1$-dimensional moduli space by gluing together these intervals along pairs that meet with opposite orientation. By considering the boundaries of this moduli space, one obtains a mod-$p$ chain homotopy between $\overline{Q_j}(C)$ and $\overline{Q_j}(C')$.
\end{proof}

\begin{definition}
\label{definition:eq-operations}
Define $$\begin{array}{l} Q: H_*(EG \times_G A) \otimes QH^*(M)^{\otimes I} \rightarrow QH^*(M). \\ Q([C] \otimes x_1 \otimes \dots \otimes x_I) := \sum_{j \ge 0} \overline{Q_j}(C)(x_1 \otimes \dots \otimes x_I) q^j, \end{array}$$ (where $q$ is the quantum variable).
\end{definition}

\subsection{Examples}
\label{subsec:examples}
Remember that we fixed $n$, so considered $\overline{\mathcal{M}}_{0,1+n}$ with marked points $z_0,z_1,\dots,z_n$ and picked $G \le \text{Sym}(n)$ and $p | |G|$. We will always fix $z_0 = 0$ on the principal component of the bubble tree.

Importantly, recall:\begin{itemize} \item $n$ is the number of marked points,
\item $G$ is the group acting on these marked points,
\item $I$ is the number of inputs of the given operation (i.e. the number of $G$ orbits in $\{1,\dots,n\}$),
\item $A$ is the space of domains (in our case, given as a selection of $z_0, z_1,\dots,z_n \in \overline{\mathcal{M}}_{0,1+n}$
\item Op is the induced operations
\item Ref is the reference
\end{itemize}

In all cases, $p$ is the characteristic. These operations are illustrated in Figure \ref{fig:table}.

\hspace*{-1.7cm}
\begin{tabular}{ |c|c|c|l|l|l| } 
 \hline
   $n$ & $G$ & $I$  & $A$ &Op & Ref \\\hline 
 $p$ & $\mathbb{Z}/p$ & $1$ &  $z_k = e^{2 \pi i k/ p}$ for $k=1,...,p$ & $x \mapsto QSt_p(x)$ & \cite{wilkins18}, \ \cite{seidelformal}\\\hline 
 $p+1$ & $\mathbb{Z}/p$ & $2$  & $\begin{array}{l} z_k = e^{2 \pi i k/ p} \text{ for } k=1,...,p,\\ \text{and } z_{p+1} = \infty\end{array}$ & $(\alpha, x) \mapsto Q \Sigma_{\alpha}(x)$ & \cite{covariant} \\\hline 
 $p+2$ & $\mathbb{Z}/p$ & $3$  &  $\begin{array}{l} z_k = e^{2 \pi i k/ p} \text{ for } k=1,...,p, \\  z_{p+1} = \infty, \text{ and } z_{p+2} \text{ varies in } S^2\end{array}$ & $(\alpha,x,\beta) \mapsto Q \Pi_{\alpha,\beta}(x)$ & \cite{covariant}\\\hline 
 $2p$ & $(\mathbb{Z}/p)_2$ &  $2$ &  $\begin{array}{l} z_{j} \text{ and } z_{p+j} \text{ have both bubbled off} \\ \text{together, for each of } j=1,\dots,p \end{array}$& $(x, y) \mapsto Q St_p(x * y)$ & \cite{wilkins18} for $p=2$\\\hline 
 $2p$ & $(\mathbb{Z}/p)_2$ &  $2$ &  $\begin{array}{l} z_{(j-1)p+1},\dots,z_{jp} \text{ have all bubbled off} \\ \text{together, for each of } j=1,2 \end{array}$& $(x,y) \mapsto Q St_p(x)* QSt_p(y)$ & \cite{wilkins18} for $p=2$\\\hline 
 $p^2$ & $(\mathbb{Z}/p)_3 \rtimes  (\mathbb{Z}/p)^p_1$ & $1$ &   $\begin{array}{l} z_{(j-1)p+1},\dots,z_{jp} \text{ have bubbled off} \\ \text{together, for each of } j=1,\dots,p \end{array}$ & $x \mapsto Q St_p \circ QSt_p(x)$ & \cite{wilkins18} for $p=2$\\
 \hline
\end{tabular}

Key: 
$$(\mathbb{Z}/p)_2 := \langle (1 \dots p)(p+1 \dots 2p) \rangle \subset S_{2p}.$$
$$(\mathbb{Z}/p)_3 := \biggr\langle kp+j \mapsto \begin{cases} \begin{array}{lr} (k+1)p+j &  \text{ if } (k+1)p+j - p^2 \le 0 \\ (k+1)p+j - p^2 & \text{ if } (k+1)p+j - p^2 > 0  \end{array}\end{cases}\biggr\rangle \subset S_{p^2}.$$
$$(\mathbb{Z}/p)^p_1 := \langle (1 \dots p),(p+1 \dots 2p),\dots,(p^2-p+1,p^2) \rangle \subset S_{p^2}.$$

We will note some precise nomenclature here, as the names are very similar: generally, $QSt_p$ is called ``the quantum Steenrod power" or ``the quantum Steenrod $p$-th power". The map $Q \Sigma_{\alpha}$ is often called ``the quantum Steenrod operation". 

\begin{figure}
    \centering
\hspace*{-2cm}\includegraphics[scale=0.45]{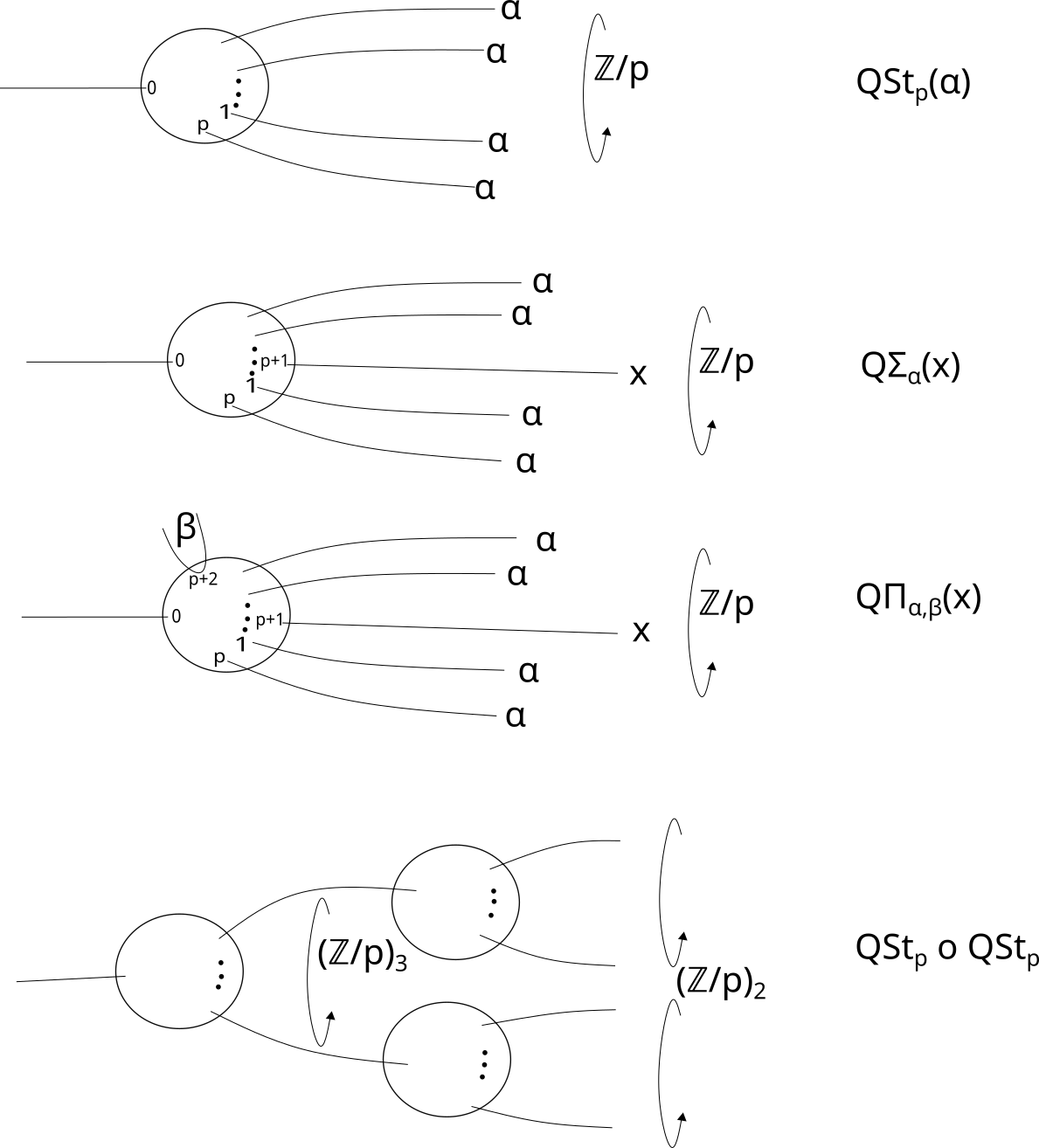}
    \caption{The first three and the last of the entries in the table above. }
    \label{fig:table}
\end{figure}

\section{Properties}

\label{sec:sec-properties}


We will now present some properties of these operations. 

\subsection{Additivity}
\label{subsec:additivity}
This is a difficult thing to prove in general, and barring a complete description of the $G$-equivariant homology of $\overline{\mathcal{M}}_{0,n}$ for all $G$ and $n$, the author is unable to determine in exactly which cases additivity can be achieved. One way to prove additivity for a specific operation -- which, we recall, is defined by taking the homology class of the homotopy quotient of some smooth closed $A \subset \overline{\mathcal{M}}_{0,n}$ -- is as follows. 

First, we find some element of $H^*(BG)$ that acts injectively via the cup product on $H^*_G(A)$. In the case where $A$ consists of fixed points, this amounts to finding such an element for $A = \text{pt}$. When for example $G$ is cyclic, there may possibly be some $H^*(BG)$-torsion elements in low degrees for which additivity fails, but additivity succeeds for everything in higher degrees. Then, if we can find such an element of $H^*(BG)$ acting injectively by the cup product on $H^*_G(A)$, one can prove additivity with details as relegated to Appendix \ref{sec:additivity}.

Despite this, we note that we can sometimes define so-called ``reduced operations", which in essence 'force' our relations to become additive. These may in general hold less information than the operations themselves, being ``modulo torsion" (indeed they may even vanish, as is always a possibility when trying to conduct localisations of a ring) but in general they may contain some information. As an example, one can use reduced operations to prove the quantum Adem relations.

Such reduced relations are defined as follows. Suppose that $t \in H^*(BG)$. Suppose also that we have fixed some space of domains $A$. We define $$H^*_{G-red,t}(A) := H^*_G(A) / \langle \ker(t_A \cup) \rangle,$$ to be the {\it operations reduced to $t$ over $A$} where as in the referenced appendix $t_A$ is the pullback of $t$ to $H^*_G(A)$ and $t_A \cup : H^*_G(A) \rightarrow H^*_G(A)$ sends $x \mapsto t_A \cup x$. In particular, if defining such relations we would inevitably like $\ker(t \cup)$ to be reasonably small (thus retaining as much information as possible). 

We note that Definition \ref{definition:eq-operations} can be dualised to yield an operation $$Q: QH^*(M)^{\otimes I} \rightarrow QH^*(M) \otimes H^*_G(A).$$ Hence, we can define an operation $$Q^{t}(X): \bigotimes_{i=1}^I H^*(M) \rightarrow QH^*(M) \otimes H^*_{G-red,t}(A),$$ by composing $Q$ with the quotient map. Moreover, we can leverage all the work in Appendix \ref{sec:additivity} to prove that such operations are additive. However, we will not prove that in this survey paper.

\subsection{Relations between equivariant operations}
\subsubsection{Methods of finding relations}
\label{subsubsec:methods-finding-relations}
Broadly speaking, relations that have been found between these equivariant quantum relations come in two flavours:
\begin{enumerate}
    \item Find two different chain-level realisations for the same element of $H^*(A 
    \times_G EG)$ in Definition \ref{definition:eq-operations} (see e.g. the Cartan relation \cite[Section 5]{wilkins18} or the covariant constantcy condition \cite[]{covariant}).
    \item Use properties, e.g. injectivity or surjectivity, of the maps associated to using different groups (see e.g. quantum Adem relations).
\end{enumerate}


Ignoring practical considerations, there are two fairly immediate, natural ways to consider relations between equivariant operations. The first is to use the equivariant topology of the space of domains, which corresponds to the first sort of relations listed above. Broadly speaking, $\mathbb{Z}/p$-equivariant topology is reasonably straightforward: repeated application of the so-called ``partial localisation" in Proposition \ref{proposition:partial-localisation} can give one a lot of information about the equivariant topology of simple spaces (like Deligne-Mumford spaces), and the work in \cite{kim} formalises this picture, at least for $\overline{\mathcal{M}}_{0,1+p}$. However, for more complicated groups $G$ this story remains very difficult. Even $G= \mathbb{Z}/p \int \mathbb{Z}/p$, which has a reasonably straightforward choice of $E \mathbb{Z}/p \int \mathbb{Z}/p = S^{\infty} \times (S^{\infty})^p$, is difficult to study by hand.

The second natural way to consider relations involves using the algebraic properties of the groups themselves. Indeed, the classical Adem relation is really a statement of combinatorics mixed with the fact that there is an injective map from $$H^*(B S_{p^2}; \mathbb{F}_p) \rightarrow H^*\left(B \mathbb{Z}/p \int \mathbb{Z}/p; \mathbb{F}_p\right).$$ In particular, suppose we are given a space consisting of elements of $\overline{\mathcal{M}}_{0,1+p^2}$, and further suppose that both $S_{p^2}$ and $\mathbb{Z}/p \int \mathbb{Z}/p$ act on this space. Then if there is some sum of elements of $\overline{\mathcal{M}}_{0,1+p^2}$ fixed by $S_{p^2}$ (or indeed a smooth manifold of fixed elements), then one can immediately apply an Adem style relation to the resulting equivariant quantum invariants.

As a final thought about abstract relations between such operations, one can ask the following question (which to the author's knowledge has not been explored): 
\begin{Question}
Does the ring structure of $H^*_G(A)$, for some space of domains $A$, say anything about the resulting operations? 
\end{Question}

\subsubsection{Finite-group localisation using cells}
\label{sec:partial-loc}
For the first type of relations we mentioned in Section \ref{subsubsec:methods-finding-relations}, in the case of using $G=\bZ/p$, we can often use the following simple yet quite powerful relation on equivariant homology:
\begin{proposition}[Partial localisation]
\label{proposition:partial-localisation}
Suppose that $\bZ/p$ acts on some space $X$ (take as some generator $g \in \bZ/p$) and $W \subset X$ is a $\bZ/p$-invariant closed manifold, with $A \subset W$ a $\bZ/p$-invariant closed submanifold of codimension $1$, with $U \subset W$ a closed submanifold of codimension $0$ such that $\partial  U = L$, and $$W = \bigcup_i g^i U.$$ Recall that we may represent generators of $H^{\bZ/p}_*(X; \bF_p)$ by elements of $C^{\text{cell}}_*(B\bZ/p;\bF_p) \otimes C^{\text{cell}}_*(X;\bF_p)$, and that the cells $\Delta_i$ from Example \ref{example:z-mod-p} generate $C_*^{\text{cell}}$ as a $\bZ/p$-module. Then the following holds:

$$[\Delta_i \otimes W] = [\Delta_{i+1} \otimes L].$$
\end{proposition}
\begin{proof}[Sketch proof]
Depending on the parity of $i$, one considers the boundary of the cellular chain:
$$\Delta_{i+1} \otimes U$$ or $$\Delta_{i+1} \otimes (g-1)^{p-2} U.$$
\end{proof}

\begin{remark}
It should be noted: we call this ``localisation" to link it e.g. to localisation by Quillen in \cite{quillen} or Atiyah-Bott localisation. In particular, this provides a way to associate a $\mathbb{Z}/p$-equivariant homology class of an embedded submanifold (or pseudocycle, etc) with the $\mathbb{Z}/p$-equivariant homology class of a subset of codimension $1$ fixed by $\mathbb{Z}/p$. In some sense, repeated application should (and conjecturally always will) allow one to identify that with the homology class associated to the fixed-point set. Hence, this can be thought of as being akin to ``localisation". However, at time of publication, for the case of $\mathbb{Z}/p$ no proof exists in writing.
\end{remark}

\subsection{Covariant constantcy, and the quantum Cartan relation}

Let us now illuminate two foundational tools in the realm of quantum Steenrod operations (although one is implied by the other). To do so, we will use what we have previously discussed -- the fact that the operations in question are determined by some choice of cycle in $H_*^G(\mathcal{M}_{0,1+n}; \mathbb{F}_p)$, and Theorem \ref{theorem:bordisms} that tells us that two homologous cycles yield identical (homology level) quantum operations. 

Throughout this section, plenty of pictures will be used. Some of these pictures will be nodal spheres with marked points, but some will have lines between spheres, or lines leaving the marked points. The rule of thumb is this: given a picture, the underlying nodal sphere (i.e. element of $\overline{\mathcal{M}}_{0,1+n}$) is obtained by shrinking all line segments to a point. To be exact, you can replace any line segment between two spheres by a node, and any line segment that touches a sphere at exactly one end can be replaced by a marked point. These lines have at times been added to link the underlying domains with the operations themselves: as an example, given a domain drawn with lines at marked point $z_i$ ending at $\alpha_i$, this will mean ``the operation whereby the input is $\alpha_1 \otimes \dots \otimes \alpha_n$". We note that the correspondence between {\it special points} and {\it lines} in our diagrams represents the fact that we can generally replace every intersection point with a perturbed Morse flowline (of finite length if it is a node, of semi-infinite length if it is a marked point), in each situation.

We will note at the end of this section that the quantum Cartan relation can be deduced from the covariant constant condition. However, because the literature already contains proofs of the covariant constant condition for all primes $p$, and the quantum Cartan relation for $p=2$, it seems reasonable that here in this survey we will sketch the direct proof of the quantum Cartan relation for primes $p>2$, using partial localisation (interested parties can compare this to the proof in \cite{covariant}).

We will first need the operations $QSt_p$, where $n=p$, $G = \bZ/p$, hence $I=1$, and $A$ is a single point where $z_0 = 0$ and the $z_1,\dots,z_p$ are the $p$-th roots of unity. We have established that this is enough to define a family of operations parametrised by $$H^{\bZ/p}_*(\{A \} \times_{\bZ/p} E \bZ/p; \bF_p) \cong H_*(B \bZ/p; \bF_p),$$ which is $1$-dimensional in each degree $*$, generated by $\Delta_i$. Alternatively, by dualising $H_*(B \bZ/p; \bF_p)$, we obtain a single operation: $$QSt_p : H^*(X) \rightarrow H^*(X) \otimes H^*(B \bZ/p; \bF_p),$$ and recall that (for $p>2$) $H^*(B \bZ/p; \bF_p), \cong \bF_p [u,\theta]/(\theta^2)$, where $|u|=2, \ |\theta|=1$.

The total operation $QSt_p$ is a map between two rings. The quantum Cartan relation asks the following question: is it a ring homomorphism?

The answer is in fact {\it no}. The reason for this is as follows: recall from \ref{subsec:examples} that we know what are the operations $(x,y) \mapsto QSt_p(x*y)$ and $(x,y) \mapsto QSt_p(x)*QSt_p(y)$, in terms of equivariant homology classes. In particular, $n=2p$, $G=\bZ/p$ (hence $I=2$), and the underlying points describing them are two different points in $\overline{\mathcal{M}}_{0,1+2p}$. Unlike in standard homology, different points could yield different (families of) elements in $H_*^{\bZ/p}(\overline{\mathcal{M}}_{0,1+2p};\bF_p)$ (one can intuit this, via Quillen localisation \cite{quillen}, through the fact that the $\mathbb{Z}/p$-fixed point set might be disconnected).  Recalling, the points in question respectively are:

\begin{itemize}
    \item For $(x,y) \mapsto QSt_p(x*y)$, when $z_j$ and $z_{p+j}$ bubble off together in pairs for $j=1,\dots,p$, where the bubble of each such pair sits at $\zeta^j$ in a principal bubble (containing $z_0$) for some primitive root of unity $\zeta$. Call this point $m_1 \in \overline{\mathcal{M}}_{0,1+2p}$, and
    \item For $(x,y) \mapsto QSt_p(x)*QSt_p(y)$, when there is a principal sphere (containing $z_0$) with two other spheres connected at nodes, such that one sphere has $z_1,\dots, z_p$ at the roots of unity and one has $z_{p+1},\dots,z_{2p}$ at the roots of unity. Call this point $m_2 \in \overline{\mathcal{M}}_{0,1+2p}$.
\end{itemize} 

The first thing to note is that there is no reason to think these points are at all related. As it happens, let $W \subset \overline{\mathcal{M}}_{0,1+2p}$ consist of points whereby on some principal component, $z_0$ is fixed at $0$ and $z_1,\dots,z_p$ are at roots of unity, and then there is a secondary sphere attached at a node, which is freely moving on the principal component but fixed at $0$ on the secondary component, such that $z_{p+1},\dots,z_{2p}$ are at roots of unity on the secondary component. This, one can see, is a $2$-dimensional submanifold of $\overline{\mathcal{M}}_{0,1+2p}$ that is diffeomorphic to $S^2$, by assigning the position of the node in the principal component. In the construction of Section \ref{sec:sec-equivariant-quantum-operations}, this is what we called ``$A$". See Figure \ref{fig:W}.

\begin{figure}
    \centering
    \includegraphics[scale=0.7]{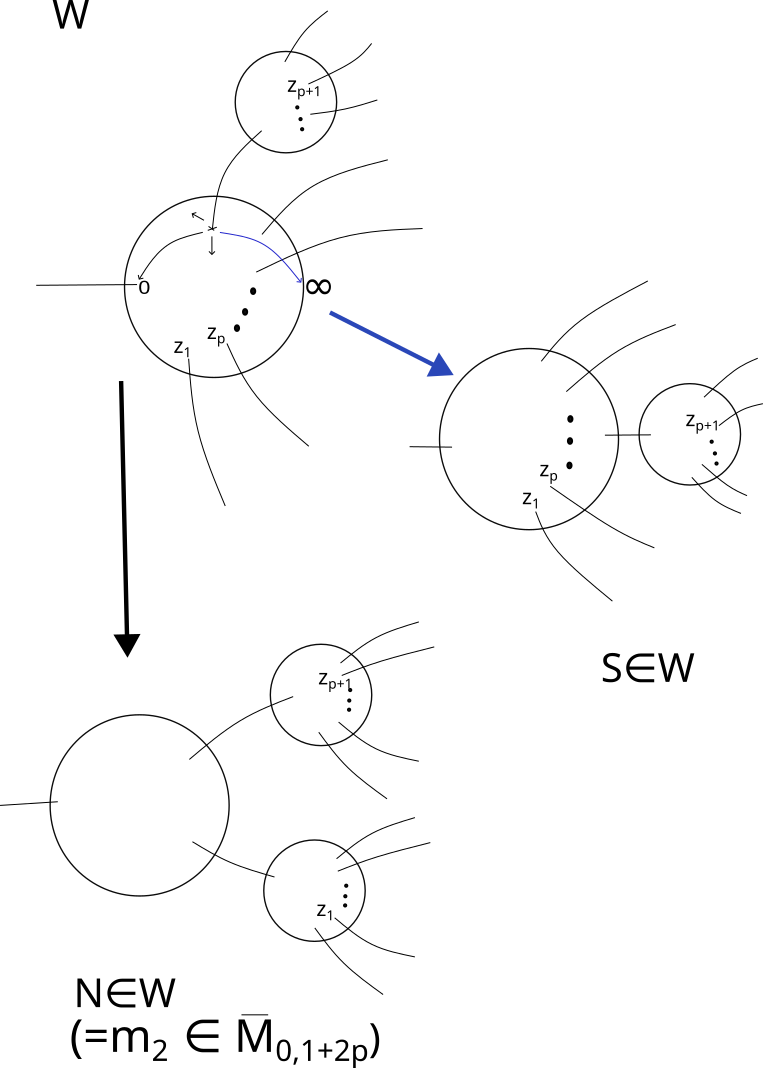}
    \caption{The space of domains $W$, and points $N,S \in W$.}
    \label{fig:W}
\end{figure}

Now observe that in $H_*^{\bZ/p}(W),$ there is a class for $i=0,1,2,\dots$ corresponding to each generator $\Delta_i \in H_{i}(B \bZ/p;\bF_p)$, with the generator associated to $i=0 $ being the fundamental class. Call these classes $[W]_i \in H_{2+i}^{\bZ/p}(W)$. There are also two families of classes associated to fixed points of the $\bZ/p$-action: one is associated to $0 \in S^2$, and one to $\infty \in S^2$.  Call these $0, \infty \in S^2$ respectively $N$ and $S$ (so in our previous notation, $[N]_i = [\{0 \} \otimes \Delta_{i}]$).

Finally, one can apply Proposition \ref{proposition:partial-localisation} twice (as an intermediate, one considers an arc of a great circle from $0$ to $\infty$ in $S^2$). One deduces that $$[W]_{i-2} = [N]_{i} - [S]_{i}.$$ (with signs depending on conventions). 

One notices that the $0 \in S^2$ corresponds to the point $m_2$, and further one observes that there is a $1$-dimensional family of $\bZ/p$-fixed points in $\overline{\mathcal{M}}_{0,1+2p}$ between the point $S$ and the point $m_1$: see Figure \ref{fig:Sequalsm1}. 

\begin{figure}
    \centering
    \includegraphics[scale=0.7]{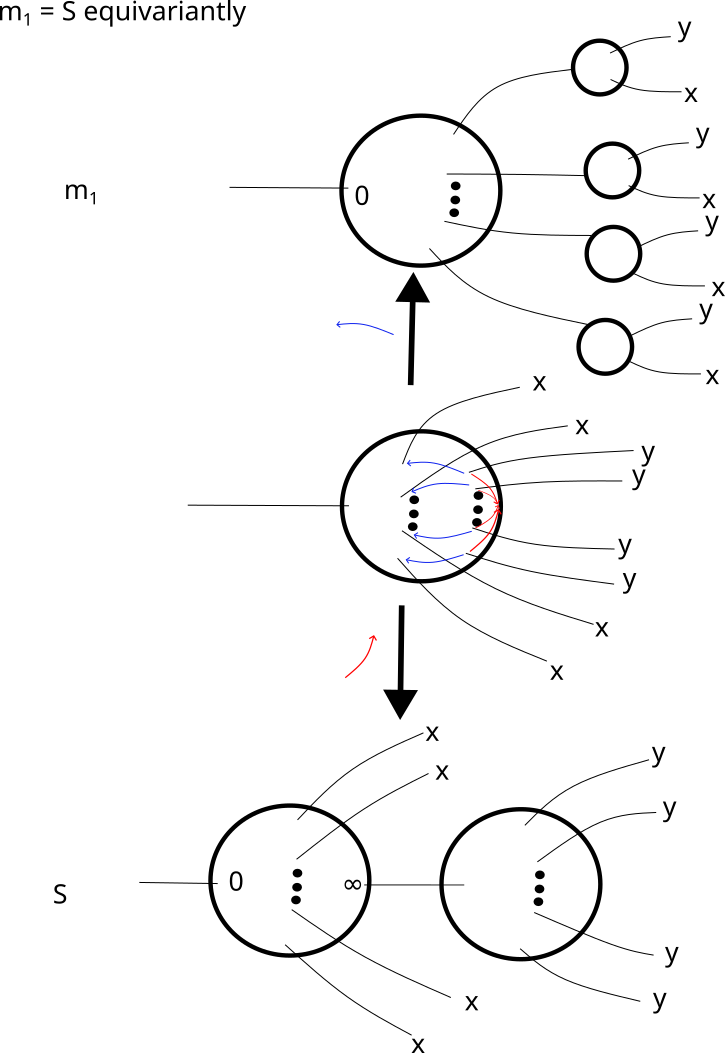}
    \caption{Proof-by-pictures that $S = m_1$ in $\overline{\mathcal{M}}_{0,1+2p}$.}
    \label{fig:Sequalsm1}
\end{figure}

One then considers a sequence of $1$-dimensional moduli spaces parametrised by chains $C_i$ such that $$dC_i = W \times \Delta_{i-2} - m_2 \times \Delta_{i} + m_1 \times \Delta_{i}.$$ By looking at the boundary of this $1$-dimensional moduli space, one deduces that there is a chain homotopy between the chain-level maps $$Q([W]_{i-2} \otimes (-)^{\otimes p}) - Q([m_2]_{i} \otimes (-)^{\otimes p}) + Q([m_1]_{i} \otimes (-)^{\otimes p}) \text{ and } 0.$$ Recalling that the $Q([m_2]_{i} \otimes (x)^{\otimes p})$  together form (up to sign) $QSt_p(x*y)$ and the $Q([m_1]_{i} \otimes (x)^{\otimes p})$  together form $QSt_p(x)*QSt_p(y)$, we thus obtain that $$(-1)^{\clubsuit} QSt_p(x*y) -  QSt_p(x)*QSt_p(y) = \sum_{ i=2j + \epsilon} Q([W]_i \otimes x^{\otimes p} \otimes y^{\otimes p}) u^j \theta^{\epsilon}.$$ We will not compute the signs here. It is worth noting that, even in the case of $\bC P^1$, this correction term $Q(W)$ can be nontrivial. 

Now we notice that the proof of the covariant constant condition in \cite{covariant} (which, to stick with the given theme, we will phrase in terms of elements of $H_*^G(\overline{\mathcal{M}}_{0,1+n})$) is essentially a slight modification of the above process. In particular, we are using $n=p+2$ instead of $2p$. We add a point $z_{p+1}$ at $\infty$ in the principal component at all points in the above, and in $W$ we replace the non-principal component with a single marked point  $z_{p+2}$. See Figure \ref{fig:constantcy}.

\begin{figure}
    \centering
    \includegraphics[scale=0.7]{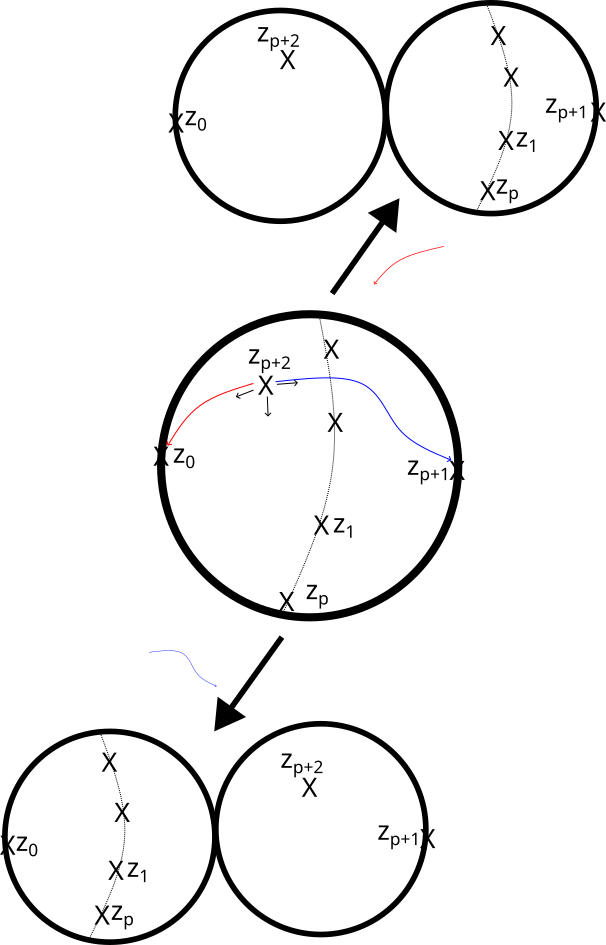}
    \caption{The domains we consider to prove covariant constantcy.}
    \label{fig:constantcy}
\end{figure}

In that instant, by considering what operations we obtain, one sees that we get that (appealing once more to Section \ref{subsec:examples}), neglecting to calculate signs, \begin{equation} \label{equation:covariant} Q\Sigma_{\alpha}(b*c) - (-1)^{\clubsuit} b * Q\Sigma_{a}(c) = Q \Pi_{a,b}(c).\end{equation} Hence, both of these relations rely upon the same underlying result in general $\bZ/p$-equivariant homology. Further, one can use the covariant constant relation \eqref{equation:covariant} to recover the quantum Cartan relation, as follows.

Firstly, one observes that $Q \Sigma_{\alpha_1 * \alpha_2} = Q \Sigma_{\alpha_1} \circ Q \Sigma_{\alpha_2}$: see Figure \ref{fig:multiplicativityiscomposition}. Next, one applies \eqref{equation:covariant} with $a = x$, $b = QSt_p(y)$ and $c = 1$. Then one obtains $$\begin{array}{ll} Q St_p (x * y) & = Q \Sigma_{x*y}(1) \\ &= Q \Sigma_{x}(Q \Sigma_{y}(1)) \\ &= Q \Sigma_x(Q St_p(y) * 1) \\&= QSt_p(y) * Q\Sigma_x(1) + Q \Pi_{x,QSt_p(y)}(1) \\ &= QSt_p(y) * QSt_p(x) +  Q \Pi_{x,QSt_p(y)}(1).\end{array}$$ The final thing to notice is that $Q \Pi_{x,QSt_p(y)}(1)(- \otimes \Delta_i) = Q([W]_i)(x,y)$, in the language above.
\begin{figure}
    \centering
    \includegraphics[scale=0.7]{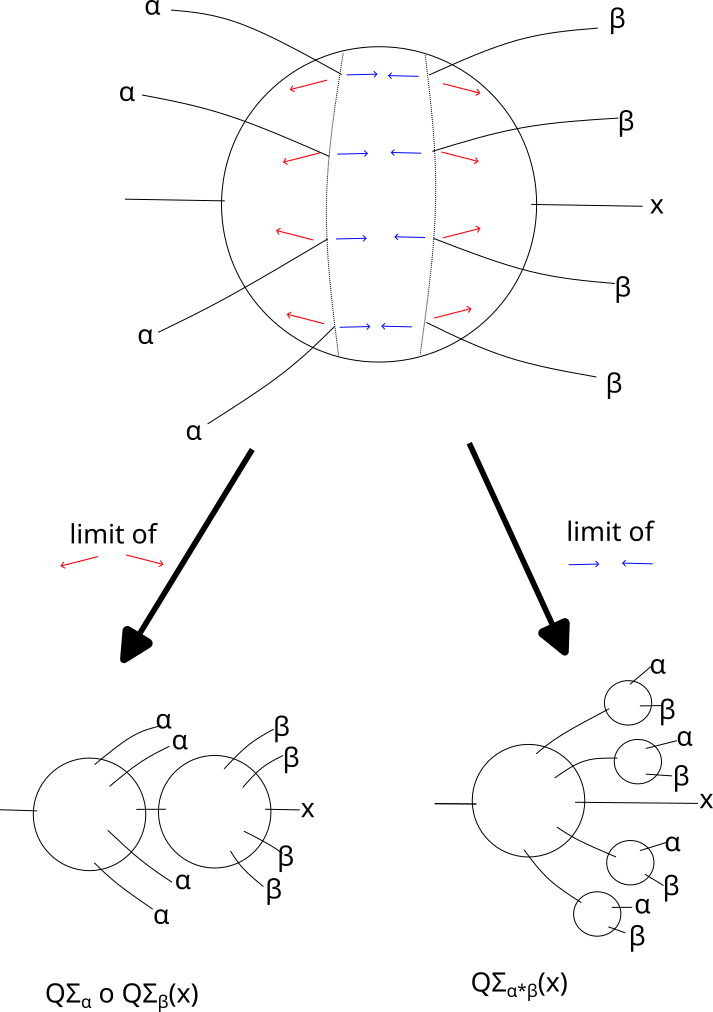}
    \caption{A demonstration that multiplicativity intertwines with composition of $Q\Sigma$.}
    \label{fig:multiplicativityiscomposition}
\end{figure}

\subsection{Quantum Adem relations}
\label{subsec:qu-ad-rel-def}
We will aim to be brief regarding the quantum Adem relations, as they are (at least in the author's previous work) technical and quite messy. However, we will summarise the main points. We notice that :

\begin{point}
\label{point:point1}
Given any $M$ equipped with a $G$-action, and the $G$-space of domains $A \subset \overline{\mathcal{M}}_{0,1+p^2}$, and $X \in H_*^G(A)$. Suppose that $H \subset G$ is a subgroup with index coprime to $p$. Then $X$ lifts to $\tilde{X} \in H_*^H(A)$ (i.e. $X = \pi_* \tilde{X}$, where $\pi: EG \times_H M \rightarrow EG \times_G M$), there is a commutative diagram:

\begin{equation}\label{classicalademdiagram}
\xymatrix{
H^*(M)
\ar@{->}^-{Q^G(-\otimes X)}[r]
\ar@{->}_-{=}[d]
&
H^*(M)
\ar@{->}^-{=}[d]
\\
H^*(M)
\ar@{->}_-{Q^H(- \otimes \tilde{X})}[r]
&
H^*(M)
}
\end{equation}

Here we have modified the notation of Definition \ref{definition:eq-operations} to include a superscript denoting which group, $G$ or $H$, to which we are referring, e.g. $Q^G$ or $Q^H$.
\end{point}

The importance of this point is as follows: $H$-equivariant operations filter through $G$-equivariant operations, in certain cases. It is critically important here that the index of $H$ in $G$ is coprime to $p$, which in turn means that the map $\pi_*$ is surjective. 

\begin{point}
\label{point:point2}
Further, Point \ref{point:point1} (and a little more work, see \cite[Section 7]{wilkins18}) implies that under the related (dualised) definition of operations, the following diagram commutes.

\begin{equation}\label{classicalademdiagram-2}
\xymatrix{
H^*(M)
\ar@{->}^-{Q^G}[r]
\ar@{->}_-{=}[d]
&
H^*(M) \otimes H^*_G(A)
\ar@{->}^-{id \otimes \pi^*}[d]
\\
H^*(M)
\ar@{->}^-{Q^H}[r]
&
H^*(M) \otimes H^*_H(A)
}
\end{equation}
\end{point}

We can, in certain cases of $A$ (e.g. the case we use below) use the notion of ``reduced operations" from Section \ref{subsec:additivity} to make things cleaner. This is just to justify the notation in \cite[Section 7]{wilkins18}. Next:

\begin{point}
\label{point:point3}
One needs to prove a combinatorial relation (for the $G=\mathbb{Z}/2$ case, see \cite[Lemma 7.1]{wilkins18}). 
\end{point}

Finally, we need to demonstrate the following claim:

\begin{point}
\label{point:point4}
Any time we have a commutative diagram as in Point \ref{point:point2}, we get a combinatorial relation (which uses Point \ref{point:point3}) between the terms of $Q^H$.
\end{point}

This point is basically an extension of \cite[Lemma 7.2]{wilkins18} (the cited reference being $p=2$ case). 

We now specialise the situation we are considering. In the case of the classical Steenrod squares, the space of domains $A$ consists of the single point, a graph with $1+p^2$ univalent vertices and a single $1+p^2$-valent vertex, connected to each univalent vertex by a single edge. We thus use $G := S_{p^2}$ and $H := \mathbb{Z}/p \int \mathbb{Z}/p$.

\begin{point}
A space of domains in $\overline{\mathcal{M}}_{0,1+p^2}$ that is fixed by $S_{p^2}$ consists of the following: let $l_1$ be the arrangement of $1+p$ spheres, attached such that there is a central sphere $S(0)$ connected to each of the others $S(1),\dots,S(p)$ with $e^{2k i\pi / p} \in S(0)$ connected to $0 \in S(k)$, and $1+p^2$ marked points $z_0,z_1,\dots$ such that $z_0 = 0 \in S(0)$ and $z_{(j-1)p+1},\dots, z_{jp}$ are associated with the roots of unity on $S(j)$. Pick coset representatives $1,g_2,\dots$ of $\mathbb{Z}/p \int \mathbb{Z}/p \subset S_{p^2}$, where the group $\mathbb{Z}/p \int \mathbb{Z}/p$ is without loss of generality generated by the cycles $((j-1)p+1,(j-1)p+2,\dots,jp)$ for $j=1,\dots,p$ and the cycle $(1,p+1,2p+1,\dots,(p-1)p+1)(2,p+2,2p+2,\dots,(p-1)p+2)\dots(p,2p,3p,\dots,p^2)$ and let $l_j = g_j \cdot l_1$ (acting by permuting the indices). Then $A = \sqcup_i \{ l_i \}$ has an $S_{p^2}$-action.
\end{point}

Hence, by all of the above, there is a combinatorial relation on the $H$-equivariant operation determined by this $G$-action. Further, $\{ l_1 \}$ is the single $H$-fixed point in the space of domains $\overline{\mathcal{M}}_{0,1+p^2}$ that, when considered as a (family of) $\mathbb{Z}/p \int \mathbb{Z}/p$-equivariant operation, defines the composition of two quantum Steenrod operations $QSt \circ QSt$. 

Hopefully it is clear that Point 1 and Point 2 is a very general fact about $H$-equivariant operations that lift to $G$-equivariant operations, and not a fact specifically about compositions of Steenrod squares. Points 3 and 4 are possibly more specific to our situation, although in reality they too might be quite general. In particular, the only distinction between the composition of the quantum and the classical Steenrod squares is the following: in the former case, the domain $l_1$ corresponding to composition of quantum Steenrod powers (i.e. Point 5) does not, by itself, lift to a $S_{p^2}$-invariant space of domains. Therefore, we obtain quantum correction terms in the Adem relations (by adding the cosets of $l_1$ under the $S_{p^2}$-action).

\begin{remark}
\label{remark:adem-example}
Setting this example for easiest case of quantum Steenrod squares, it would require pulling $QSt \circ QSt$ back to become a relation over $S_4$. However, here one runs into a problem because this is not the case. One can see this through the lens of the equivariant symplectic operad principle. We would like to say the following: to consider maps over $S_4$, one would like to consider the $S_4$-equivariant homology of Deligne-Mumford space. The domain of $QSt \circ QSt$ is the $1+4$-pointed nodal sphere, consisting of marked points $z_0$ on a central sphere, which is attached to two spheres, one with $z_1$ and $z_2$, the other with $z_3$ and $z_4$. Then this is a fixed point under the $D_8 = \langle (12), (34), (13)(24) \rangle$-action, hence determines a collection of cycles in $H_*^{D_8}(\overline{\mathcal{M}}_{0,1+4})$. This point is not, however, fixed under the $S_4$-action, and so does not necessarily determine a cycle in $H_*^{S_4}(\overline{\mathcal{M}}_{0,1+4})$.

However, the sum of (the $D_8$-relation) $QSt \circ QSt$ and some other operation {\bf can} be lifted to $S_4$. The other operation consists of the sum of the points $a$ and $b$, which are domains in $\overline{\mathcal{M}}_{0,1+4}$ where respectively $(z_1,z_2)$ comes together with $(z_3,z_4)$ and $(z_1,z_3)$ comes together with $(z_4,z_2)$. Compare with the final Point above. Note that neither of $a$ nor $b$ defines a $D_8$-equivariant cycle, but the two together do. Similarly, the underlying domain of $QSt \circ QSt$ union with $a$ and $b$ defines a $S_4$-equivariant cycle, but no subset does likewise.

This $D_8$-equivariant operation (using the union of the domains $a$ and $b$), provides a correction term to the naive quantum Adem relations, and it is provably nonzero using a calculation on $\mathbb{CP}^2$ (see \cite[Example 1.6]{wilkins18}, which we prompt the reader of in Section \ref{subsec:ad-examples}). 
\end{remark}

\section{Quantum Steenrod powers and quantum Adem relations - calculations}
\label{sec:calculations}
We now present the methods of calculating equivariant quantum operations as detailed in \cite{wilkins18}, and \cite{covariant}. In particular, with the terminology as in the beginning of Section \ref{sec:eq-qu-op}, we take $n=p$ prime, $G= \bZ/p$, and $A$ to comprise of the points in $\mathcal{M}_{0,1+p}$ where $z_0=0$ and $z_1,\dots,z_p$ are the $p$-th roots of unity. We can see that these are all of the $p-1$ (up to reparametrisation) fixed points of $\mathcal{M}_{0,1+p}$ under the $\bZ/p$-action.

Then we can construct operations as in Definition \ref{definition:eq-operations}, one for each point $a \in A$, and it can be seen that for some choice of point in $A$ the resulting operations $QSt_p$ defined by $QSt_p(x) = Q([a] \otimes x)$ is the quantum Steenrod $p$-th power as in \cite{covariant} (up to a sign).

Importantly, because $E \bZ/p \times_{\bZ/p} A \cong (B \bZ/p)^{p-1}$, we may use Example \ref{example:z-p-additive} to conclude that these operations are additive.

\begin{remark}
Using the results in \cite{kim}, we observe that the operation above contains most of the information associated with the operation using $A=\mathcal{M}_{0,1+p}$, except for non-equivariant Gromov-Witten invariants. 
\end{remark}

\subsection{Calculating these operations}

Considering that these operations often involve making a generic choice of $S^{\infty}$-dependent Morse function, it would seem at first glance that these sort of operations are difficult to calculate. However, it actually turns out that we can get quite far by leveraging Proposition \ref{proposition:partial-localisation} in context. We will demonstrate the most powerful computational tool, the covariant constancy relation, although we remark here that in the case of toric varieties one can use the so-called Cartan relations to similar effect (the proof of the quantum Cartan relations being functionally identical to the proof of covariant constancy, and indeed the former being a consequence of the latter).

The statement of the covariant constant relation is:
\begin{equation}\label{equation:covariant-constant} Q\Sigma_{a}(x * \omega) -\omega* Q\Sigma_{a}(x) = d_{\omega} Q \Sigma_{a} x,\end{equation} for any $\omega \in H^2(M).$ The reason why this is so useful is the following algorithmic procedure: 

\begin{itemize}
    \item We begin with $M$ monotone, so choose $\omega = c_1$.
    \item Suppose that we restrict attention to spheres with $c_1 = \lambda$.
    \item Then we get the equation:
    $$\sum_{\mu_1 + \mu_2 = \lambda} Q\Sigma^{\mu_1}_{a}(x *_{\mu_2} \omega) -\omega *_{\mu_1} Q\Sigma^{\mu_2}_{a}(x) = \lambda Q \Sigma^{\lambda}_{a} x,$$ where here $*_{\mu}$ is the quantum product using spheres of Chern $\mu$, and $Q \Sigma^{\lambda}$ is the operation $Q \Sigma$ using curves of Chern $\lambda$.
    \item We notice that $$\sum_{\mu_1 + \mu_2 = \lambda} Q\Sigma^{\mu_1}_{a}(x *_{\mu_2} \omega) = Q\Sigma^{\lambda}_{a}(x \cup \omega) + \sum_{\mu_1 + \mu_2 = \lambda, \ \mu_1 < \lambda} Q\Sigma^{\mu_1}_{a}(x *_{\mu_2} \omega).$$
    \item We use \eqref{equation:covariant-constant} with $x$ replaced by $x \cup \omega$, but still restricting to spheres of Chern $\lambda$. If $p \nmid \lambda$ we can write over $\bF_p$ that: $$  Q \Sigma_{a} (x \cup \omega) = \tfrac{1}{\lambda} \bigl(   Q\Sigma_{a}(x \cup \omega * \omega) -\omega* Q\Sigma_{a}(x \cup \omega) \bigr).$$
    \item Now we iterate this process, observing the following: we iteratively obtain terms of the form: $\tfrac{1}{\lambda^j} Q\Sigma^{\lambda}_a (x \cup \omega \cup \dots \cup \omega)$, where the cup product is taken $j+1$ times. Notice that eventually this terminates, once $j$ is sufficiently large.
    \item Repeat this by noticing $$\sum_{\mu_1 + \mu_2 = \lambda} \omega *_{\mu_1} Q\Sigma^{\mu_2}_{a}(x) = \omega \cup Q\Sigma^{\lambda}_a(x) + \sum_{\mu_1 + \mu_2 = \lambda, \ \mu_1 > 0} \omega *_{\mu_1} Q\Sigma^{\mu_2}_{a}(x).$$ 
    \item After sufficiently many iterations, we obtain $\lambda Q \Sigma^{\lambda}_a(x)$ in terms of $Q \Sigma^{\mu}_a(y)$ where $\mu < \lambda$ and $y$ is obtained from $x$ by quantum multiplication. Considering that $Q \Sigma^0_a(x)$ is the classical Steenrod square, this means that if we know (A) the Steenrod square and (B) the quantum cohomology, we can inductively compute $Q \Sigma^{\lambda}_a(x)$ for every $a$, $x$ and $\lambda < p$. 
\end{itemize}

 This follows a general trend, which one can see going all the way back to our discussion wherein operations are determined from equivariant cohomology of the domains of our moduli spaces: in some sense, $p$-fold covered curves (which might occur once we allow Chern number $p$) provide all of the ``interesting" contributions to the $p$-th quantum Steenrod power operation, because their calculation inherently involves some equivariant geometry (namely the normal bundle of the space of multiply covered curves). Contrast this to spheres of Chern $<p$ where, in essence, everything is determined by the equivariant algebraic topology and the nonequivariant quantum cohomology.

In Section \ref{subsec:jhlee-mcc}, we provide a brief overview of the work by J. H. Lee. The referenced work is the first calculation of some non-trivial multiply covered quantum Steenrod contributions (and the only such calculation, to the knowledge of the author, at time of writing).

\begin{task}
Calculate the contributions to some Steenrod $p$-th power arising from $p$-fold covers of spheres.
\end{task}

Of course, in the monotone case, $p$-fold covers are noticeable by their absence because we exclude them from our moduli spaces. Even so, they have an impact on the equivariant cohomology. The question is, what impact do they have?

\subsection{A final word on the quantum Adem relations}
\label{subsec:ad-examples}

An interesting line of work involves the so-called quantum Adem relations; we only have a proof of the quantum Adem relations for the quantum Steenrod square, but the proof itself should pass through (with a suitable factorial rise in complexity).

\begin{figure}
    \centering
\includegraphics[scale=0.7]{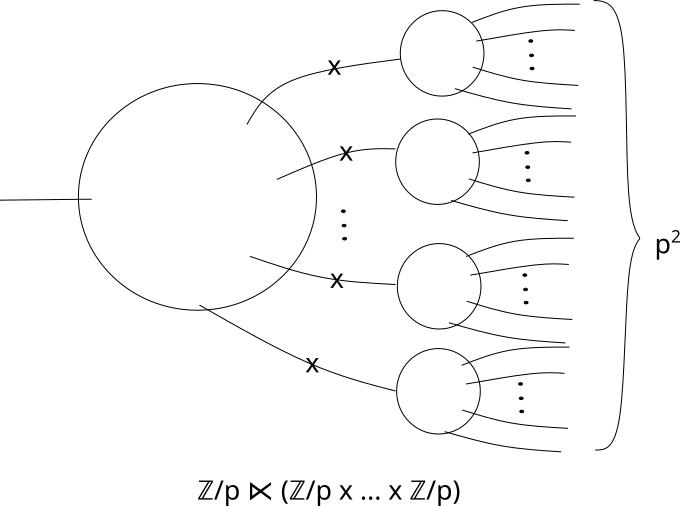}
    \caption{The domain and group action used to define $QSt_p \circ QSt_p$.}
    \label{fig:domain-and-group-action}
\end{figure}

Classically, the Adem relations deal with what happens if one composes Steenrod operations. Trying to distill further Section \ref{subsec:qu-ad-rel-def}, the composition of Steenrod squares, $Sq \circ Sq$, is naturally an operation over $D_8 = \bZ/2 \int \bZ/2$ (and similarly for Steenrod $p$-th powers over $\bZ/p \int \bZ/p$).  See Figure \ref{fig:domain-and-group-action} for the domains. However, because $Sq \circ Sq$ also naturally exists as an operation over $S_{p^2}$, and bearing in mind that $H^*(B S_{p^2}) \xhookrightarrow{} H^*(B (\bZ/p \int \bZ/p))$, this inclusion alone provides constraints to what format this operation can take.

\begin{Question}
The Adem relations for the quantum Steenrod square come about because of the inclusion homomorphism $D_8 \rightarrow S_4$. Are there other group homomorphisms $\phi: H \rightarrow G$ that determine an interesting constraint on $H$ or $G$ operations?
\end{Question}

To see an example of quantum Adem relations whereby the correction term (in the language of Remark \ref{remark:adem-example}) associated to $a+b$ is nontrivially, one does not need to go further than $\bC P^2$. We will not replicate the example here, because the calculation involves defining some technical language that is not particularly enlightening. However, firstly we recall the standard Adem relations on the Steenrod squares $Sq$. 

				Given $x \in H^{*}(M)$ and $\alpha,\beta>0$ such that $\beta<2\alpha$,

				\begin{equation} \label{equation:classicalademrelations} Sq^{\beta}Sq^{\alpha}(x) = \sum_{k=0}^{[\beta/2]} {{\alpha-k-1}\choose{\beta-2k}}Sq^{\alpha+\beta-k} Sq^{k}(x).\end{equation}

Now observe that there is a natural definition of $QSt_2^{\beta}$ and $QSt_2^{\alpha}$ (i.e the supercript $\alpha$ refers to using as parameter space $S^{|x|-\alpha} \subset S^{\infty} = E \bZ / 2$). Similarly, there is a natural guess for the quantum version of the Adem relations (replacting $Sq$ with $QSt_2$ above). 

However, in the case of $\bC P^2$, using $x \in H^2(\bC P^2;\bF_2)$ as the generator and $\alpha = \beta = 2$, the natural extension of \eqref{equation:classicalademrelations} is incorrect by a factor of $q^N$ (the quantum variable raised to the minimal Chern number). For those wanting to fully understand the terminology, \cite{wilkins18} contains the relevant excruciating detail.

As a side-note, this is interesting viewed through the lens of finite-group localisation (as in e.g. Quillen \cite{quillen}): in particular, we know that for cyclic groups, the equivariant cohomology of a space is equal (up to torsion) to the fixed point set. Some might ask the question of whether this result holds for other groups. Here we see the fixed point set of $\overline{\mathcal{M}}_{0,1+4}$ under the $D_8$-action is a single point (the underlying domain of $QSt \circ QSt$), which defines a family of elements of $H_*^{D_8}(\overline{\mathcal{M}}_{0,1+4}; \bF_2)$, corresponding to $QSt \circ QSt$. However, the sum of points $a+b$ also determines a nontrivial element in $H_*^{D_8}(\overline{\mathcal{M}}_{0,1+4}; \bF_2)$. We do not know if this operation can be computed through Steenrod power operations, if it is torsion, or if neither are true. If neither is true, then we have essentially demonstrated that there is no localisation result for $D_8$ (not, of course, that we would expect one in the first place).

\section{Floer invariants}
\label{sec:sec-floer-invariants}

We have so far discussed equivariant operations on quantum cohomology, $QH^*(M)$. However, it is well known that the PSS map of \cite{PSS} yields a map $PSS: QH^*(M) \rightarrow HF^*(H)$ for any Hamlitonian $H: M \rightarrow \mathbb{R}$ (we will omit the definition of $HF^*(H)$, directing readers towards e.g. \cite{hofersalamon}. In general, to digest this section a reader would be expected to be reasonably confident with Floer theory). In certain cases, such as when $M$ is compact or $H$ is $C^2$-small, this $PSS$ map is an isomorphism. So it is natural to ask how these equivariant quantum operations interact with the PSS map.

Indeed, after the initial work by Fukaya \cite{fukaya}, it was such Floer operations that were first studied in greater depth, in \cite{seidel}. There, Seidel defined an operation analogous to the Steenrod square, but for fixed-point Floer cohomology (which we will also not define here). In \cite{shelukhin-zhao} this was extended to general Steenrod powers, and in \cite{wilkins2} this was connected (for $p=2$) to the quantum Steenrod square under an equivariant PSS map (in upcoming work, joint between this author and E. Shelukhin, a similar result to \cite{wilkins2} will be proven for general $p$).

We will explore the cited work in a small amount of detail in Section \ref{subsec:fpfloer}, but for now we will discuss fundamental structural differences between the previous work and similar ideas in Floer theory.

When looking at $G$-equivariant Floer theoretic operations, there is an added complication: this is because we can think of the space of domains (which is what we are interested in, by the equivariant symplectic operadic principle) as being built over Deligne-Mumford space with marked points and a choice of direction -- for the point-at-infinity of the cylindrical end -- at each marked point. This means that the correct way to define equivariant operations in our usual way is to build them over the appropriate $(S^1)^{1+n}$-bundle over $\mathcal{M}_{0,1+n}$. What one finds, however, is that this adds a particular rigidity to our construction that is unavoidable in general. For example, the quantum Cartan relation fails to hold in the same way in the case of Floer theory, specifically because the underlying chain relation does not lift from $H_*^{\bZ/p}(\mathcal{M}_{0,1+n})$ to a chain-level relation on this $(S^1)^{1+n}$-bundle. As a brief aside, there \textit{is} a quantum Cartan relation, but it is fundamentally weaker than that in the quantum case, being more a relationship on the level of modules than rings.

All is not lost, however. One can certainly define equivariant symplectic operations in the same way, and using this augmented space of domains. Indeed, much the same, any set of domains fixed under this $G$-action will determine a symplectic operation. One must just be aware that, due to the $S^1$ at each marked point, there are now {\it fewer} fixed sets, hence less operations. 

We will not define these operations in general here in this survey, due to the number of requisite technical details. However, we will observe the following important points.

\subsection{Space of domains}

Following everything we have done so far, including the explicit construction in Section \ref{sec:sec-equivariant-quantum-operations}, we see that the important information one needs to consider when defining an equivariant invariant is the space of domains. It is reasonably tricky to see how a general permutation group $G$ might act on the aforementioned $(S^1)^{1+n}$-bundle over $\mathcal{M}_{0,1+n}$. The author does not have a solution for this in general, although the cyclic bar construction might be useful in this. 

In the instance where $G$ is a cyclic group, and when we restrict the $(S^1)^{1+n}$-bundle to some $X \subset \mathcal{M}_{0,1+n}$ such that the generator of this cyclic group acts via a biholomorphic map, then we can define the action on the bundle-restricted-to-$X$ to be the differential on of this holomorphic map at each of the $S^1$-fibres. 

As such, in the case where we want to define analogues of the Steenrod $p$-th power operation for Floer cohomology, the appropriate point in the $(S^1)^{1+n}$-bundle is the following: the group is $G = \mathbb{Z}/p$. The base in $\mathcal{M}_{0,1+n}$ consists of the first point at $0$, and the rest at roots of unity in some order (just as for defining the Steenrod power operations). We observe that the generator of $\mathbb{Z}/p$ acts via a rotation $R$, and so we fix an asymptotic point for the $0$ and $1$ vertices, the latter being say $\theta_1$, and then define the asymptotic point $\theta_i$ at the $i$ vertex to be $\theta_i = d R^i (\theta_0)$.

Note however that this is {\it not} a fixed point, because $dR$ will act on the asymptotic point at $0$. At first this may be concerning, but we notice that $dR^p = Id$, and hence this gives us a way to fix things: by changing the codomain of the Floer Steenrod $p$-th power operation, as we shall see in the next section.

\subsection{Codomain of these operations}

Recall for example quantum $G$-equivariant operations on a manifold $X$, which always land in $$H^*(BG \times X) \cong H^*(BG) \otimes H^*(X).$$ Because of the fact that $dR$ acts nontrivially on the ``outgoing" asymptotic point for the vertex labelled $0$, in order to define the Floer Steenrod $p$-th power operation, we will need that the codomain of the Floer Steenrod $p$-th power operation is some analogue of a homotopy quotient (i.e. analogous to $H^*(EG \times_G X)$ with a nontrivial action of $G$ on $X$). In fact, the operation has the following domain and target:
\begin{equation} \label{equation:pst} \mathcal{P}: H^*_{\mathbb{Z}/p}((CF^*(H))^{\otimes p}) \rightarrow HF^*_{ eq}(p \cdot H),\end{equation} for any Hamiltonian $H$. Here, on the left hand side we use the $\mathbb{Z}/p$-action that cyclically permutes the $CF^*(H)$ in the tensor product. On the right hand side is a version of $\mathbb{Z}/p$-equivariant homology for Floer theory. We will not give a full rigorous definition, but we will say that:
\begin{enumerate}
    \item the chains are $CF^*(p \cdot H)[[u]] \otimes \Lambda(e),$ where $u$ and $e$ should be compared to the generators of $H^*(B \mathbb{Z}/p; \mathbb{F}_p)$ for $p>2$, and we compare to classical equivariant chains, we have replaced polynomials over $u$ with power series over $u$,
    \item the differential is of the form $d_{eq} = d_0 + d_1 + \dots$, an infinite sequence of linear maps (well defined because we work over power series in $u$), where $d_0$ is the regular Floer differential and $d_1$ is an approximation to the second term in $d_{\mathbb{Z}/p}$, see Section \ref{subsubsec:equiv-cohom}.
\end{enumerate}

The reason for using this version of equivariant cohomology, and not the standard $\mathbb{Z}/p$-equivariant cohomology, is as follows: in general our defining auxiliary data $(J,H)$ must not have any sort of symmetry with respect to the $p$-legged pair-of-pants across which $J$ and $H$ vary (as this would break regularity). However, in order to define $HF^*_{ eq}(p \cdot H)$ as $HF^*_{\mathbb{Z}/p}(\text{some chain complex})$, we would like to use the chain complex $CF^*(p \cdot H)$ (it is the only 'obvious' option). Yet we necessarily cannot allow the almost complex structure to have a $\mathbb{Z}/p$-symmetry while remaining regular, and therefore $\mathbb{Z}/p$ cannot act on the holomorphic curves used to define the differential of $CF^*(p \cdot H)$. In particular, there is not a $\mathbb{Z}/p$-equivariant differential, which we would need to define standard $\mathbb{Z}/p$-equivariant cohomology.

We will note that one can go a small step further, using the fact that the map $$\text{power}: C \rightarrow C^{\otimes p}_{\mathbb{Z}/p},$$ $$\text{power}(c) = c \otimes \dots \otimes c,$$ is well defined on homology, and thence one can define $$\mathcal{P}St_p = \mathcal{P} \circ \text{power}.$$ Unlike $QSt_p$, for example, these $\mathcal{P}St_p$ may not have  nice properties such as additivity.

\begin{remark}
This gives us a potential way to define equivariant operations more generally: if for example we continue to use  $dR$, but if it acts such that $\theta_1 \neq dR(\theta_0)$, then the domain of such an equivariant operation is not just $H^*_{\mathbb{Z}/p}((CF^*(H))^{\otimes p})$ using cyclic permutation, but rather $H^*_{g}((CF^*(H))^{\otimes p})$ using an action $g$ that does not just permute the $CF^*(H)$ terms but also rotates them. However, this is speculative and has not been greatly explored.
\end{remark}

\subsection{Relation to quantum invariants}

In the case where $G = \mathbb{Z}/p$, for small Hamiltonians $H$ there is similarly an operation $PSS_{eq}$ generalising the PSS-map, from equivariant quantum cohomology $QH^*(M) \otimes H^*(BG)$ to Floer cohomology $HF^*_{eq}(H)$. Indeed, the arguments used in the non-equivariant case demonstrate that this must be an isomorphism. In this world, ``the choice of asymptotic point at $0$" somehow looks like it sits on the boundary of an ``equivariant disc", and using e.g. localisation, see Section \ref{sec:partial-loc}, one sees that this is the same as sitting at the centre of said disc (with a shift in equivariant degree). For those familiar with $\mathbb{Z}/p$-equivariant homology, this is the intuition as to why this equivariant PSS-map generally intertwines equivariant quantum and equivariant symplectic invariants. See Figure \ref{fig:pssintertwine}. The result itself, which will appear in work \cite{shelwilk} by E. Shelukhin and this author (and is proved for $p=2$ in \cite{wilkins2}), is: $$ \mathcal{P} St_p \circ PSS = PSS_{eq} \circ QSt_p.$$

\begin{figure}
    \centering
\hspace*{-1cm}\includegraphics[scale=0.65]{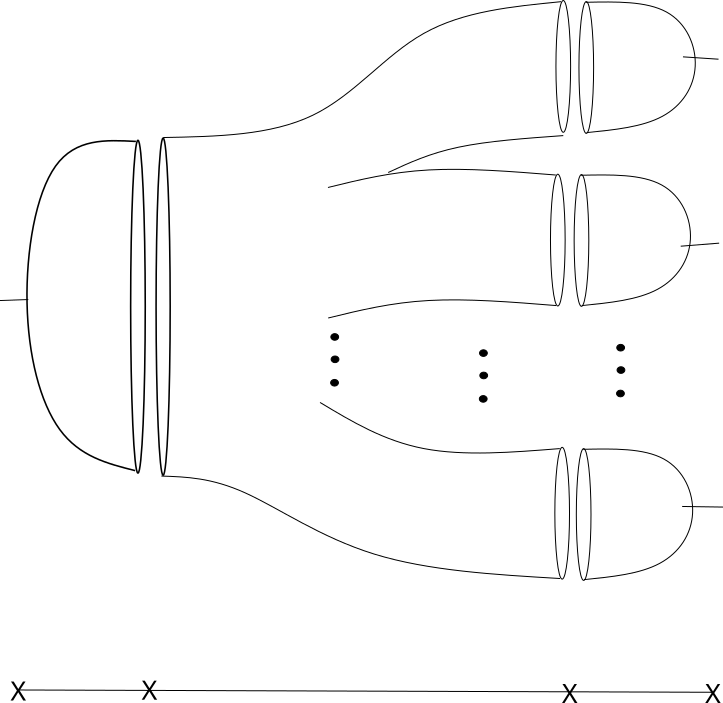}
    \caption{The PSS and equivariant PSS intertwining relation.}
    \label{fig:pssintertwine}
\end{figure}

\section{Recent (as of 2024) work by other authors}
\label{sec:recentwork}

\subsection{Pseudorotations}
\label{pseudorotations}
First we will define (mod-$2$) pseudorotations. We will initially use the definition \cite[Definition 1.1]{cgg}. We are working with $\mathbb{F}_2$-coefficients here, and we will be using the classifying space of $\mathbb{Z}/2$, which has homology $H^*(B \mathbb{Z}/2; \mathbb{F}_2) \cong \mathbb{F}_2 [h]$, where $|h| = 1$.

A Hamiltonian diffeomorphism $\phi$ of a closed symplectic manifold is called a {\it mod-$2$ pseudorotation} if for all $N \in \bZ_{\ge 1}$, the map $\phi^N$ is nondegenerate and the Floer differential (with $\bF_2$-coefficients) vanishes for $\phi^N$. 

One is often interested in the spectral invariants of a Hamiltonian diffeomorphism. Given some filtered chain complex $(C,\mathcal{A})$ (where $C$ is the chain complex and $\mathcal{A}$ is the function giving the filtration), one can define for $\alpha \in H^*(C)$ the spectral invariant $$c(\alpha,(C,\mathcal{A})) := \inf \{a \in \mathbb{R} | \alpha \in \text{im}(H(C)^{<a} \rightarrow H(C)) \} \cup \{ \infty \}.$$ In our case, we are interested in the Hamiltonian Floer cohomology associated to some Hamiltonian diffeomorphism $\phi$, and $\mathcal{A}$ is the usual action functional. One can think of it as the ``first action level at which the cohomology class $\alpha$ appears". 

Importantly, one can also define a spectral invariant $\hat{c}$ for the filtered equivariant Floer cohomology. The crucial result for the following subsubsections is as follows: the equivariant pair-of-pants product, \eqref{equation:pst}, is not just an operation on from $HF^*_{\mathbb{Z/p}}$ to $HF^*_{eq}$, but rather from {\it filtered} $HF^*_{\mathbb{Z}/p}$ to {\it filtered} $HF^*_{eq}$. In particular, we know something about the action values of the outputs relative to the inputs. 

Note that, in order to facilitate a survey of 

\subsubsection{Approach 1 -  Shelukhin}

The work in the papers by Shelukhin is the starting point of the notion of ``$\mathbb{Z}/2$-Steenrod uniruled". 

 Consider $PD(\text{pt})$, a generator of $H^{2n}(M^{2n})$. Then, the condition $$``QSt(PD(\text{pt})) \neq h^{2n} PD(\text{pt})",$$ (i.e. the quantum Steenrod square of the highest degree cohomology class is perturbed) is known as ``$\mathbb{Z}/2$-Steenrod uniruled", to distinguish from the traditional notion of uniruled (uniruled being that any point has a holomorphic curve through it: notice the similarity between this notion and $QSt(PD(\text{pt})) \neq h^{2n} PD(\text{pt})$).

 The statement is then as follows: if $(M^{2n},\omega)$ is a closed symplectic manifold equipped with a mod-$2$ pseudorotation, then the manifold is ``$\mathbb{Z}/2$-Steenrod uniruled". To give the theorem:

 \begin{theorem*}[See \cite{egorshel1}, Theorem A]
For a closed, monotone symplectic manifold $(M, \omega)$ satisfying the Poincar\'e duality property, the existence of a mod-$2$ pseudo-rotation implies that $(M, \omega)$ is $\mathbb{Z}/2$-Steenrod uniruled.
 \end{theorem*}

 In particular, ``Poincar\'e duality" is a relation between the spectral invariants associated to $[M]$ and $[pt]$ with respect to symplectomorphisms $\phi$ and $\phi^{-1}$, which we will not elucidate here to avoid having to define spectral invariants.

One should consider these as a result towards the notion of uniruledness (hence the name ``Steenrod uniruled"). In particular, the only way that the quantum Steenrod square of $PD(\text{pt})$ may be perturbed is if there are some Gromov-Witten invariants through a ``$\mathbb{Z}/2$-equivariant point", by which we mean a Gromov-Witten invariant intersecting some small sphere twice (this small sphere being considered as the unit normal bundle of a point).

In the initial paper, \cite{egorshel1}, this required the addition of an extra condition; a Poincar\'e duality condition. In the second paper, \cite{egorshel2}, this extra condition was eliminated, fully generalising the result.

To prove this, one needs to make energy estimates. Suppose that there is such a pseudorotation, $\phi$. One can demonstrate that the equivariant spectral invariant of $\hat{c}(QSt([pt]),\phi^2)$ is exactly twice $c([pt],\phi)$. If $QSt([pt])$ were undeformed, i.e. $QSt([pt])  = [pt] h^{\text{dim}(M)}$, then that looks like comparing $2c([pt],\phi)=\hat{c}([pt],\phi^2)$, which is at least $c([pt],\phi^2)$ (this requires some work to show, i.e. a choice of Morse function with unique minimum, energy estimates and so forth). Thus, $c([pt],\phi^{2k}) \le 2^k c([pt],\phi)$, i.e. sublinear growth. One then uses another method (e.g. Poincar\'e duality in \cite{egorshel1}) to imply strictly-superlinear growth for some choice of $k$. This inequality generally does not rely on the fact that there is a pseudorotation. In \cite{egorshel2}, the need for Poincar\'e duality was avoided using a combinatorial argument.

\subsubsection{ Approach 2 - \c{C}ineli-Ginzberg-G\"urel}

\cite{cgg} runs is along the similar lines to that of Shelukhin, i.e.  suppose $(M^{2n},\omega)$ is a closed symplectic manifold equipped with a mod-$2$ pseudorotation. Then one wants to prove $M$ is not $\mathbb{Z}/2$-Steenrod uniruled.

The main theorem in question is:

\begin{theorem*}[See \cite{cgg}, Theorem 1.2]
For a closed monotone symplectic manifold $(M^{2n}, \omega)$, the existence of a non-degenerate Hamiltonian mod-$2$ pseudo-rotation implies that the quantum Steenrod square $\mathcal{QS}$ of $PD(pt) \in H^{2n}(M; \mathbb{F}_2)$ is deformed.
\end{theorem*}

An argument by contradiction is used here. One assumes both that there is a pseudorotation, and that $QSt(pt)$ is undeformed. The crux of the proof is that if $QSt(pt)$ is undeformed, then it is divisible by $h^{2n}$. Thus, so too must be the associated equivariant pair-of-pants product applied to the Floer cohomology class $PSS_{\phi}(pt)$. But what $PSS_{\phi}$ is doing is considering $pt$ as a (capped) orbit of the Hamiltonian flow induced by $\phi$. One can likewise consider the equivariant $PSS_{eq,\phi^2}(pt)$, viewing $pt$ as a (capped) orbit of the Hamiltonian flow induced by $\phi^2$, and the highest order in $h$ term will be, at the chain level, $PSS_{\phi^2}(pt)$. Then, if there is an index jump from the $\phi$ to the $\phi^2$ case, then this provides a contradiction because one has that $h^{2n}$ divides $QSt(PSS_{\phi}(pt))$ but $QSt(PSS_{\phi}(pt)) = PSS_{eq,\phi^{2}} \circ \mathcal{P}(pt) = h^m PSS_{\phi^2}(pt) + O(h^{m-1})$ for some $m<2n$. Finally, one demonstrates that there is an index jump for $\phi^{2^k}$ for some $k$.

\subsubsection{Steenrod uniruledness implies uniruledness - Rezchikov}

This result in \cite{semon} is important contextually to the above, in that it demonstrates the following: if a monotone or minimal Chern $> 1$ symplectic manifold is $\mathbb{Z}/p$-Steenrod uniruled for infinitely many primes $p$ (this naturally extends the definition of ``$\mathbb{Z}/2$-Steenrod uniruled"), then its quantum product (over $\mathbb{Q}$) is deformed. There is a similar result for a non-monotone symplectic manifold with minimal Chern number equal to $1$, although in this case all one needs is $\mathbb{Z}/2$-Steenrod uniruled. 

The Theorem in question is the following (noting here that the notation $\mathbb{F}_2$-uniruled is the same as $\mathbb{Z}/p$-Steenrod uniruled from earlier papers):

\begin{theorem*}[See \cite{semon}, Theorem 1]
 If a $2n$-dimensional symplectic manifold $(M, \omega)$
$$\begin{cases} \begin{array}{l}\text{is monotone and is }  \mathbb{F}_p \text{-uniruled for infinitely many primes } p, \text{ or} \\ 
\text{has minimal Chern number } N>1 \text{ and is } \mathbb{F}_p \text{-uniruled for infinitely many primes } p, \text{ or} \\ \text{has minimal Chern number } N=1 \text{ and is } \mathbb{F}_2 \text{-uniruled}, \end{array} \end{cases}$$
 then there is some quantum product that is not the same as the cup product.
\end{theorem*}

This result connects the results in Section \ref{pseudorotations} to those traditional results in uniruledness, in particular demonstrating that the name ``Steenrod uniruled" is apt. Further, it allows one to connect these mod-$p$ pseudorotations, which only store mod-$p$ information, with a result in characteristic $0$. This can be seen as a missing link between the work by the previously mentioned authors, and the Chance-McDuff conjecture. 

The proof uses techniques such as the covariant constant equation of \cite{covariant}, some structural properties in \cite{seidelformal}, and more observations about the structure of quantum Steenrod operations by Rezchikov. We will not elaborate on the method of proof here.

\subsection{Formal groups - P. Seidel}

This work by P. Seidel in \cite{seidelformal} is among the more structural. It details the fact that the space of Maurer-Cartan elements on a symplectic manifold $X$ (which, in the context of an $A_{\infty}$-algebra are solutions of an equation involving the $A_{\infty}$-operations $\mu_i$) has a formal group structure (we do not include the definition in this survey). 

What this then allows is for one to define $p$-th power maps of Maurer-Cartan elements. One can then demonstrate that the space of Maurer-Cartan elements projects down to the odd cohomology of $X$, and that projection of the $p$-th power map of Maurer-Cartan elements covers the quantum Steenrod $p$-th power operation.

\subsection{Calculations of multiply covered curves - J. H. Lee}
\label{subsec:jhlee-mcc}

In \cite{jae1}, J. H. Lee performs the first true calculation of contributions to quantum Steenrod powers $QSt_p$ via multiply covered curves. The author does this by adapting the work of Voisin, \cite{voisin}, to the equivariant case, including equivariant with respect to an $S^1$-action on the underlying symplectic manifold (hence, all cohomology in the main theorems is taken with respect to $\mathbb{Z}/p \times S^1$, the former group acting on the domain of the pseudoholomorphic maps and the latter group acting on the target). It should be emphasised that, at time of writing, this is the only calculation of some sort of quantum Steenrod operation using holomorphic spheres with Chern number divisible by $p$. 

The main result is:

\begin{theorem*}[See \cite{jae1}, Theorem 1.2 (and Theorem 4.12 and Corollary 6.10 for the calculation)]
Let $M$ be the cotangent bundle of $\mathbb{P}^1$. The $S^1$-equivariant quantum Steenrod operation is a covariantly constant endomorphism for the equivariant quantum connection. Its classical term (the coefficient of $q^0$ from the Novikov ring) agrees with the cup product with the classical Steenrod power $St(b)$ of the class $b \in H^2_{S^1} (T^* \mathbb{P}^1 ; Fp)$, Poincar\'e dual to a cotangent fiber.

This endomorphism can be computed in all degrees, i.e. for any $p > 2$ one can compute the coefficients for all
$\mathbb{Z}/p$-equivariant parameters $t^k \theta^{\epsilon}$, and for all powers of the $S^1$-equivariant parameter $h$.
\end{theorem*}

The idea of this work, as understood by the author of this survey, is the following: one is considering the total space of a line bundle over $\mathbb{P}^1$. In \cite{voisin}, one uses an 'atypical' compactification of the space of pseudoholomorphic maps, which matches the standard count of curves on the open stratum, splitting a perturbed holomorphic sphere in the total space into a horizontal sphere (i.e. section of the bundle) and a vertical holomorphic sphere. After demonstrating regularity and so forth, this new compactification is used to calculate contributions of multiply covered curves to curve counts. In \cite{jae1}, this construction/compactification is combined with existing studies of quantum Steenrod operations, including extensions of properties such as covariant constantcy, to provide a way to count contributions arising from multiply covered curves in the quantum Steenrod power operations. 

\subsection{QSt and $p$-curvature for symplectic resolutions - J. H. Lee}
\label{subsec:jhlee-curv}

In \cite{jae3}, J. H. Lee relates the Steenrod power operations $Q \Sigma_b$ to the ``$p$-curvature" for symplectic resolutions. We note that in \cite{jae3}, the operations and the underlying quantum cohomology are different from the operations and cohomologies listed in this survey. In particular, they are taken to be equivariant with respect to a cocharacter of a torus action, acting on the underlying symplectic manifold. 

In brief, if one studies pseudoholomorphic maps $u:S \rightarrow M$, then so far in this survey we have studied equivariance with respect to an action on $S$, whereas in \cite{jae3} one upgrades these notions to include a Hamiltonian torus action $T$ on $M$. 

We briefly note that the $p$-curvature with $b \in H^*(M)$, $F_b$, for a mod-$p$ quantum connection $\nabla_b$, is $F_b := \nabla^p_b - t^{p-1} \nabla_b$. We note that $t \in H^2(B \mathbb{Z}/p; \mathbb{F}_p)$ is an equivariant parameter. Then the main theorem is as follows:

\begin{theorem*}[See \cite{jae3}, Theorem 1.2]

For almost all primes $p$, the $T$-equivariant quantum Steenrod operation and $p$-curvature with $b \in H^*(M)$ agree.
    
\end{theorem*}

Broadly, the idea is to first study the difference between these operators, and demonstrate that this difference is nilpotent. This is what involves the representation theory, although a key part of this proof involves studying in detail how each of these two operators interacts with the ``shift operator" (see \cite[Section 2.4]{jae3}, although there are references provided to the ancestors of the ideas in this paper). The last step (showing that this difference in fact vanishes for most $p$) requires a consideration of the spectrum of this difference, viewed as a linear operator. 

\subsection{Relation to work on Lagrangian Floer cohomology - Z. Chen}

In \cite{zihong}, Z. Chen gives a first idea stage of the theoretical basis of quantum Steenrod operations, by relating them to operations on Hochschild Cohomology. By moving into the realm of Langrangian Floer cohomology, this firmly entrenches quantum Steenrod operations as an underlying structure across symplectic cohomology.

There are a lot of results in \cite{zihong}, but we would like to focus here on the first main result (chronologically within the paper), as follows: there is a relation between the quantum Steenrod $p$-th power operations, and a $\mathbb{Z}/p$-equivariant cap product on $\mathbb{Z}/p$-equivariant $HH^*$. This relation is achieved via the $\mathbb{Z}/p$-equivariant open-closed map, $OC^{\mathbb{Z}/p}$. The author is not sufficiently familiar with the material, at time of writing, to give any intuition regarding the proof. However, following is the main result for our purposes:

\begin{theorem*}[See \cite{zihong}, Theorem 1.3]
Let $F$ be the Fukaya category of $M$. Then the following relation holds (for any $b \in QH^*(M)$): 

$$Q \Sigma_b \circ OC^{\mathbb{Z}/p} (-) = OC^{\mathbb{Z}/p} \circ \left( CO(b) \cap^{\mathbb{Z}/p}- \right),$$ as maps from $HH^{\mathbb{Z}/p}_*(F) \rightarrow QH^{*+p |b| +n}(M)[[t,\theta]]$.
 
\end{theorem*}

We provide no definition of the Fukaya category of $M$ in this survey. We will simply state in the above that $HH$ denotes the Hochschild homology, and $HH^{\mathbb{Z}/p}$ the $\mathbb{Z}/p$-equivariant Hochschild homology. Further, $\left( CO(b) \cap^{\mathbb{Z}/p}- \right)$ is a $\mathbb{Z}/p$-equivariant capping operation with the class $CO(b) \in HH_*(M)$, which is the closed-open map of $b$: see \cite[Equation (1.2)]{zihong}.

\subsection{Relation to quantum Kirwan operations - G. Xu}
\label{subsec:xu}

In \cite{xu}, G. Xu compares the quantum Steenrod operations $Q \Sigma_b$ on $QH^*(M // G)$ to the classical Steenrod operations $\Sigma_b$ on $H^*(BG)$, for a symplectic manifold $M$ with a Hamiltonian $G$-action (along with certain assumptions on $M$ and $G$). One does this comparison through the construction of ``equivariant quantum Kirwan maps", from $H^*(BG)$ to $QH^*(M // G) \times H^*(B \mathbb{Z}/p).$ The main usefulness of this equivalence is in computation, providing a new avenue to compute quantum Steenrod operations: the classical Steenrod operations are reasonably easy to compute, and the (equivariant) quantum Kirwan map can be calculated through the formalism of ``affine vortices", which we will not expand upon here.

The particular theorem is as follows (note that the hypotheses listed in the theorem are technical hypotheses we will not include here, but for a little context have titles such as ``regular quotient", ``equivariant convexity", ``contractible target" and ``equivariant montonicity"):

\begin{theorem*}[See \cite{xu}, Theorem B]
Under certain hypotheses, \cite[Hypotheses 1.1-1.4]{xu}, there is a ``$\mathbb{Z}/p$-equivariant quantum Kirwan map",
$$\kappa^{eq}: H^*(BG) \rightarrow QH^*(M // G) \times H^*(B \mathbb{Z}/p).$$ such that for all $x,b \in H^*(BG; \mathbb{F}_p)$, $$Q \Sigma_{\kappa(b)}(\kappa^{eq}(x)) = \kappa^{eq}(\Sigma_b(x)).$$
    
\end{theorem*}

To explain some of the notation, $\kappa$ is a non-equivariant quantum Kirwan map (defined in \cite[Theorem A]{xu}), $V$ is symplectic manifold with a Hamiltonian $K$-action and $V // K$ is the symplectic quotient. The map $Q \Sigma$ is the quantum Steenrod operation, as in \cite{covariant} and Section \ref{subsec:examples}, and $\Sigma_b$ is the associated classical Steenrod operations (contrast this with the notation in Section \ref{subsec:jhlee-curv} wherein this represents the quantum Steenrod operation).

The argument proceeds by considering moduli spaces of holomorphic curves associated to the glued operations corresponding to $Q \Sigma_b \circ \kappa^{eq}$ and $\kappa^{eq} \circ \Sigma_b$, and taking resulting $1$-dimensional moduli spaces to define chain-homotopies between the operations.

\subsection{Usage in proving the exponential-type conjecture - Z. Chen}
\label{subsec:zihong2}

Due to the recentness of the paper \cite{zihong2}, little will be said about this result in the survey. However, the author would like to note that Z. Chen has proved a conjecture (stated in \cite{kkp} and \cite{ggi}), which was originally proved in \cite{pomerleanoseidel} with more strict conditions, in a very succinct way using quantum Steenrod operations. In particular, the conjecture states that for any monotone symplectic manifold, the {\it quantum $t$-connection} of \cite{pomerleanoseidel} has a decomposition into connections with simple poles at $0$ and monodromy being roots of unity (i.e. has exponential type). 

This author's understanding of the proof is as follows: one reduces the problem (through nontrivial means that we will however not explore here) to calculating that some operation involving the $p$-curvature is nilpotent. Then, using ideas surveyed in Section \ref{subsec:jhlee-curv} of this article (with the idea for the extension to this case attributed to P. Seidel), one can reduce this exponential-type conjecture to proving that the sum of the $p$-curvature plus the quantum Steenrod operation is nilpotent. One then proves this nilpotence by demonstrating compatibility between the quantum Steenrod operations and the decomposition of cohomology via eigenvalues of the operator ``multiplication by $c_1$".

\subsection{Fixed point Floer homology power operations - P. Seidel, and then E. Shelukhin and J. Zhao}
\label{subsec:fpfloer}

In the work of \cite{seidel} and \cite{shelukhin-zhao}, a similar invariant to the quantum Steenrod power operation (respectively for $p=2$ and general $p$) is explored. We include it here, even though these are not operations on quantum cohomology, to give some context. The construction in \cite{seidel} is the precursor to the quantum Steenrod square, and the operations detailed in the two papers are related to the quantum Steenrod powers via an equivariant PSS map (see the next section, Section \ref{sec:sec-floer-invariants}). 

Given some symplectomorphism $\phi$, these papers (roughly) explore bounds on $\dim HF^*(\phi^p)^{\mathbb{Z}/p}$ with respect to $\dim HF^*(\phi)$, which should be viewed as some sort of localisation result. They do this through making action estimates, which are used to understand how these power operations (see \eqref{equation:pst})) act on filtered and local fixed point Floer homology. We will not cover these papers in more detail in this survey.

\section{$S^1$-equivariant quantum operations}
\label{sec:s1-equ}
We would be remiss not to mention that the ideas above make some amount of sense for $S^1$ (or tori). We will not go into great detail about what exactly we mean by ``some amount of sense", except by giving a couple of examples.

\subsection{$S^1$-equivariant Seidel map - Todd Liebenschutz-Jones}

Work has been done in this realm by Liebenschutz-Jones, considering $S^1$-equivariant quantum operations. It should be pointed out that, although the result in \cite[Theorem 1.7]{todd} was proven independently, it can be viewed as an $S^1$-version of the covariant constantcy condition of \cite{covariant} using (instead of Section \ref{sec:partial-loc}) the Section \ref{subsec:LbP}. This can be realised as a special case of the result in the following subsection. In particular, it can be seen as yet another instance of a relation arising from considering the $G$-equivariant homology of the space of domains.

\subsection{Pseudocycle localisation}
\label{subsec:LbP}
We give the intuition for a method to conduct localisation using pseudocycles, appearing in the paper by this author \cite{LbP}. 

Broadly, the idea is the following. Given some smooth manifold $X$ with a smooth $S^1$-action, the fixed point set $F$ is a union of even codimension submanifolds. Atiyah-Bott localisation \cite{atiyah-bott} is a relationship between $H^*_{S^1}(X)$ and $H^*_{S^1}(F) \cong H^*(F) \otimes H^*(B S^1) \cong H^*(F)[u]$ for $u$ representing the generator of $H^2(BS^1)$. For our purposes, we will need that the $S^1$-action is semifree (although this is not a requirement in \cite{atiyah-bott}).  

Quillen localisation \cite{quillen} is the equivalent of Atiyah-Bott localisation for $\mathbb{Z}/p$. Our interpretation of Atiyah-Bott localisation will be as Proposition \ref{proposition:partial-localisation} is to Quillen localisation. 

Given some closed $N \subset X$ such that $S^1$ acts on $N$, there is a homology class $[N_i] \in H_{\text{dim}(N) +2i}^{S^1}(X)$ corresponding to the pushforward of the fundamental class under: $$N \times_{S^1} S^{2i+1} \xhookrightarrow{} X \times_{S^1} S^{\infty}.$$ The statement is the following:

\begin{theorem}[LbP]
    \begin{equation} \label{equation:LbP} [X]_{i-1} = \sum_{\alpha : \text{codim}(F_{\alpha}) = 2} \pm [F_{\alpha}]_{i} \pm [B]_{i},\end{equation} where $\{ F_{\alpha} \}$ are the connected components of $F$ and $B_{i+1}$ is defined below (and the $\pm$ can be determined).
\end{theorem}

The proof is broadly the following: consider $W = X \setminus TF$, where $TF$ is some tubular neighbourhood of $F$. Then $\pi: W \rightarrow W / S^1$ is a principal $S^1$-bundle. We know that this can be trivialised by removing some $B' \subset W / S^1$ of codimension $2$. Hence, if $TB'$ is a tubular neighbourhood of $B'$, then there is a section $$s: (W / S^1) \setminus TB' \rightarrow W \setminus \pi^{-1}(TB').$$ Finally, the map $$\left( (W / S^1) \setminus TB'\right) \times D^{2i} \rightarrow \left( W \setminus \pi^{-1}(TB') \right) \times D^{2i} \rightarrow X \times S^{\infty} \rightarrow X \times_{S^1} S^{\infty},$$ is (after removing some codimension $2$ pieces) a pseudocycle bordism demonstrating \eqref{equation:LbP}, where:

\begin{itemize}
    \item $B$ is the image of $TB'$,
    \item $D^{2i} \subset S^{2i+1} \subset S^{\infty}$ is a natural $2i$-disc that is the closure of a fundamental domain for the $S^1$-action on $S^{2i+1}$,
    \item the first map is $s \times id$,
    \item the second map is the product of the blow-down map (collapsing the fibres of the tubular neighbourhood boundaries) with inclusion,
    \item the third map is projection.
\end{itemize}

The power of this method is in the following setup: suppose we are given a universal equivariant moduli space $\mathcal{M}_{eq}$ (with elements of the form $(v,m,u)$ for $v \in S^{\infty}$, $m \in \overline{\mathcal{M}}_{0,k}$ for some $k$, and $u$ a perturbed pseudoholomorphic map with data dependent on $v$, and some evaluation conditions dependent on $m$). Then there is a projection $$\rho: \mathcal{M}_{eq} / S^1 \rightarrow S^{\infty} \times_{S^1} \overline{\mathcal{M}}_{0,k}.$$ We can apply localisation by pseudocycles to $S^{\infty} \times_{S^1} \overline{\mathcal{M}}_{0,k}$ and, assuming $\rho$ is sufficiently nice, pull back the relevant pseudocycles and the pseudocycle bordism between them under $\rho$. This allows us to compare equivariant symplectic invariants as if we were applying a localisation theorem to spaces of domains. Presumably, this method would work more broadly.

\appendix
\section{Mod-$p$ peudocycles}
\label{sec:AppA}
This will proceed much as in the case of standard arguments over $\bZ$ or $\bZ/2$-coefficients.

\begin{remark}
We will use the notation ``c-pseudocycles" rather than ``pseudocycles" because these objects are not necessarily as general as one would like. Any notion of a mod-$p$ pseudocycle should definitely include all mod-$p$ c-pseudocycles, but there may be a more general notion (the ``c" is indicative that we require the boundaries are the images of covering maps). Importantly, the mod-$p$ c-pseudocycles have the relevant important properties that we need to rigorously define things in this paper.
\end{remark}

\subsection{Mod-$p$ c-pseudocycles}

\subsubsection{The definition}
\label{sec:sec1}
The issue with defining mod-$p$ pseudocycles is that they should be based on manifolds with boundaries that ``vanish modulo $p$". This may require one to consider manifolds with corners, and the suchlike. To avoid this, we use a definition that works in our context.

Let $X$ be a smooth oriented $n$-dimensional manifold. We recall that a set $A \subset X$ is ``of dimension $\le l$" if there are a finite collection of maps, their union surjecting onto $A$, with the domains being smooth manifolds of dimension $\le l$.

A $k$ mod-$p$ c-pseudocycle is defined to be a map $f: M \rightarrow X$ such that $M$ is a smooth oriented $k$-dimensional manifold with boundary, $\partial M  = D$, and the omega-limit-set of $f$ is in codimension $2$. Further, either $D = \emptyset$ or:
\begin{enumerate}
  \item there is some smooth $\overline{D} \supset D$ such that $\overline{D} \setminus D$ is of dimension $\le k-2$, in addition to an extension of $f|_D$ to a $(k-1)$-pseudocycle $\overline{f}: \overline{D} \rightarrow X$, with $\overline{f}(\overline{D} \setminus D)$ being contained in the union of the images of some $(k-2)$-pseudocycles,
  \item there is some $\check{D}$ and a $p$-to-$1$ smooth covering map $\overline{D} \rightarrow \check{D}$, such that $f$ descends to a pseudocycle $\check{f} : \check{D} \rightarrow X$ (so i.e. $\check{f} \circ c = \overline{f}$).
  \end{enumerate}
  
\subsubsection{Mod-$p$ c-bordisms}
\label{sec:sec2}
Let $f^i:M^i \rightarrow X$ be $k$ mod-$p$ c-pseudocycles for $i=0,1$. A {\it $(k+1)$ mod-$p$ c-bordism} between $f^0$ and $f^1$ is a $k+1$-dimensional manifold $B$ with boundary $$(M^0 \sqcup (- M^1) ) \cup E,$$ and codimension $2$ corners $$(M^0 \cap E) \sqcup ((- M^1) \cap E),$$ along with a map $h:B \rightarrow X$. Further,

\begin{itemize}
    \item $M^i \cap E = \partial M^i =: D^i$ for $i=0,1$ (with notation as in property $(2)$ of Section \ref{sec:sec1}),
     \item $h|_{M^i} = f^i$ for $i=0,1$,
     \item the omega-limit-set of $h$ is in codimension at least $2$.
 \end{itemize}

 Further, either $E = \emptyset$ or the following holds (taking notation from Pseudocycle properties $(1)$ and $(2)$ in Section \ref{sec:sec1}): 
 \begin{enumerate}
     \item $E$ extends to some smooth $\overline{E}$ such that $\overline{E} \setminus E$ is in dimension $\le k-1$, and there is some $\overline{h}: \overline{E} \rightarrow X$ that is a bordism between the $\overline{D^0}$ and $\overline{D^1}$, satisfying that $\overline{h}(\overline{E} \setminus E)$ can be covered by a union of mod-$p$ c-bordisms of dimension $\le k-1$,
     \item this $\overline{E}$ is a smooth $p$-to-$1$ covering $d: \overline{E} \rightarrow \check{E}$, and there is a $\check{h}: \check{E} \rightarrow X$ so that $\check{h}$ is a bordism between the covering bases $\check{D^i}$ of the covering spaces $\overline{D^i}$, with $\overline{h} = \check{h} \circ d$.
 \end{enumerate}

We can form an obvious equivalence relation e.g. by saying $f^0 \sim f^1$ if there is some $f^t$ for $t = \{0, 1/n, \dots, (n-1)/n, 1 \}$ such that $f^{i/n} \sim f^{(i+1)/n}$ for $i=0,\dots, n-1$. Alternatively, one can combine these $n-1$ pseudocycle c-bordisms into a single one through gluing and smoothing, to obtain a group like $\mathcal{H}_*(X)$ as in \cite{zinger}.

 \subsubsection{Intersections}
 One needs to show that all the stuff you expect to hold with respect to bordisms does indeed hold. 

 In particular, given two transversely intersecting mod-$p$ c-pseudocycles $e,f$ of complimentary dimension, there is a well-defined intersection number $e \cdot f \in \mathbb{Z}/p$. The key point is that, given transversality and the fact that the boundaries are almost pseudocycles themselves, when counting $e \cdot f$ the only possible intersection points comes from the interiors of the manifolds. We know that the number of intersection points is dimension $0$ (e.g. preimage theorem). To see that in fact the set of such points is compact, this is identical to classical pseudocycles (i.e. the omega limit set is in codimension $2$). 

 Further, if $f^0: M^0 \rightarrow X$ and $f^1:M^1 \rightarrow X$ are mod-$p$ bordant via $h: B \rightarrow X$, then $e \cdot f^0  = e \cdot f^1.$ To see why this is, observe that $e \cap h$ defines a $1$-dimensional smooth manifold with boundary (assuming transverse intersections). In particular, suppose $e:N \rightarrow X$. Let $N^o$ be the interior of $N$. If one considers $e|_{N^o}$ (the restriction of $e$ to the open stratum) intersected with $h|_{B^o}$, this defines a $1$-dimensional smooth manifold, with boundary corresponding to $N^o \cap \partial B \cup \partial N \cap B^o$. As with standard theory, the signed count of the endpoints of this moduli space vanishes, hence vanishes modulo $p$.

 Assuming that the pseudocycle $e|_{\overline{\partial N}}$ intersects transversely with the bordism $h|_{\overline{\partial B}}$, this intersection is trivial (for dimension reasons). Hence, if one compactifies the $1$-manifold of $e|_{N^o} \cap h|_{B^o}$ then one obtains elements of $e|_{\partial N} \cap h|_{B^o}$ or $e|_{ N^o} \cap h|_{\partial B} $.

 Consider first $e|_{\partial N} \cap h|_{B^o}$. By the assumptions in Section \ref{sec:sec1}, given generic choices, for degree reasons $h|_{B^o}$ intersects trivially with $e|_{\overline{\partial N} \setminus \partial N}$. Hence $e|_{\partial N} \cap h|_{B^o} = e|_{\overline{\partial N}} \cap h|_{B^o}$. Recall that $e|_{\overline{\partial N}} = \check{e} \circ c$, where $c: \overline{\partial N} \rightarrow \check{\partial N}$, as exists because $e$ is a mod-$p$ c-pseudocycle. Hence $e|_{\overline{\partial N}} \cap h|_{B^o} = \check{e} \circ c \cap h|_{B^o}$, which is divisible by $p$ hence vanishes modulo $p$. 

 Thus, modulo $p$ it is sufficient to state that the sum of the remaining endpoints, those corresponding to $N^o \cap \partial B$, vanish modulo $p$. We know that $\partial B = M^0 \cup (-M^1) \cup E$. Hence, $e \cdot f^0 - e \cdot f^1 + e \cap h|_E$ vanishes modulo $p$. To see that $e \cap h|_E$ vanishes modulo $p$, we notice that this equals $e \cap \overline{h} = e \cap \check{h} \circ d$, which vanishes modulo $p$. Hence $e \cdot f^0 = e \cdot f^1$ modulo $p$, as required.

 \begin{remark}
 The important thing to note here is that there are no corners, i.e. intersections of some boundary of $B$ with the boundary of $e$. Hence, the intersection of $h$ and $e$ is some $1$-dimensional manifold with boundary, and these boundaries consist of $e \cdot f^0, - e \cdot f^1$ and points that come in families of order $p$ (through the existence of a $p$-to-$1$ covering map).
 \end{remark}

 \subsubsection{A note on definitions and the intuition behind them}
 Let us first put here why we use the definition above: suppose one takes a free chain complex over $\bZ$, called $(C,d)$. Then $a \in C$ is closed in the $\bF_p$ homology $H_*(C,\bF_p)$ if and only if $da = p b \in C$ for some $b \in C$. Further, $p \cdot db = d(pb) = d^2 a = 0 \in C$, so $db = 0 \in C$. In particular, we have {\bf on the chain level}, that a closed mod-$p$ chain $a$ satisfies that $da$ is a $p$-fold cover of some closed (in characteristic $0$) chain $b$ (in our circumstances, closed characteristic $0$ chains are represented by genuine pseudocycles, hence our condition on $D$ in Section \ref{sec:sec1}). This provides a justification for the definition of mod-$p$ c-pseudocycles.
 
Suppose instead we attempt to define a ``weak mod-$p$ c-pseudocycle", which didn't need to have its boundary being a $p$-fold covering of some pseudocycle. Instead, suppose its boundary is only bordant (not necessarily equal) to a $p$-fold cover. Relating to the above, this would correspond to some $a \in C$ such that $da = b' \in C$ and $b' = pb + dc$. But if we want that $a$ is closed over $\bF_p$, then over characteristic $0$, we obtain that $p$ divides $b'$, so $p$ divides $dc$, and $da = p(b+ c')$, where $pc' = dc$: hence, this reduces to the case that we have already seen.

Given the above, let us move on to discuss mod-$p$ c-bordisms in Section \ref{sec:sec2}.  

 Unlike mod-$p$ c-pseudocycles, we can reinterpret mod-$p$ c-bordisms in the following way: suppose that $da = p b$ and $d a' = p b'$. To prove that $a \sim a'$ in mod-$p$ homology, one finds $c$ such that $a - a' = dc$ mod-$p$. This is equivalent to saying that $a - a' = dc + p e \in C$ for some $e \in C$. But then $p de = p b - p b' \in C$. In particular, $de = b - b'$. So $e$ here is a ``bordism" between the ``primitive boundaries" of $a$ and $a'$. But more generally, if $d c' =  p e - E$ then $d(c+c') = a - a' - E$. So remembering $c+c'$ (the weak bordism), and $c'$ (the link between the weak bordism and the strong bordism), recovers $c$.

 \subsubsection{Morse mod-$p$ c-pseudocycles}
 Firstly, we fix our closed smooth manifold $X$, with a Morse function $f:X \rightarrow \bR$. Given some $a \in CM^*(M,f; \bF_p)$ with $d_p a=0$ (we use $d_p$ as the differential over $\bF_p$-coefficients), we will construct a mod-$p$ c-pseudocycle associated to $a$ in $M$. 

 In general, given some closed element $a \in CM^*(M^{\times p},f; \bF_p)$, we observe that (lifting up to $\mathbb{Z}$) $d a = p b$ for some $b \in C^*(M^{\times p},f;\mathbb{Z})$. Further, observe that $db = 0$. From $a$ we can form a mod-$p$ c-pseudocycle as follows. We proceed as in \cite{schwarzmorsesingiso}, taking the unstable manifolds of the critical points that appear in $a$, one can gluing their faces in pairs (with opposite orientation) except we will have left over $p$ boundaries corresponding to the unstable manifolds arising from critical point in $b$. This glued object will be the definition of $M$ from the first section above, and $f$ will be the evaluation map. Suppose that we let $b = \sum_{j=1}^k b_j$ for $b_j \in \text{crit}(f)^{\otimes p}$. Then the edges of $M$ correspond to having $p$ copies $W_{j,1}, \dots, W_{j,p} = W(b_j,f)$, where $W(b_j,f)$ is the product of the stable manifolds of the critical points arising in $b_j$. Then let $D_i$ be obtained by taking the union of the $W_{b_1,i}, \dots, W_{b_k,i}$. Because $b$ is closed over $\mathbb{Z}$, if one compactifies each of the $W_{j,1}, \dots, W_{j,p}$ (with objects of dimension at most $|a|-2 = |b|-1$), for each $j$, gluing along respective edges, one obtains $p$ copies of the Morse pseudocycle associated to $b$. The covering map then just collapses these $p$ copies into one. This demonstrates that the construction of Schwarz \cite{schwarz1999equivalences} does indeed yield a mod-$p$ c-pseudocycle, with this glued object being the domain of the evaluation map. A similar argument demonstrates that a mod-$p$ Morse boundary between $a$ and $a'$ yields a mod-$p$ c-bordism. 

 We thus from the above deduce the following lemma:

 \begin{lemma}
 \label{lemma:mod-p-pseudos}
 Given a closed smooth manifold $M$, every $x \in HM_*(M,f;\bF_p)$ is represented by a unique (up to mod-$p$ c-bordism) mod-$p$ c-pseudocycle.
 \end{lemma}

 \begin{corollary}
 \label{corollary:corollary-2}
 Using the $EG$ from Lemma \ref{lemma:lemma-1}, every $X \in H_*(BG; \bF_p)$ may be represented by some mod-$p$ c-pseudocycle (up to mod-$p$ c-bordism).
 \end{corollary}
 \begin{proof}
 Combine Lemma \ref{lemma:mod-p-pseudos} and Corollary \ref{cor:corollary-1}.
 \end{proof}

 \begin{remark}

 Further, every dimension $n-k$ mod-$p$ c-pseudocycle $h$ yields an element of $HM^k(M,f;\bF_p)$, defining $\Psi_h: HM_k(M,f;\bF_p) \rightarrow \bF_p$ by $\Psi_h(a) = h \cdot g_a$. This is well defined up to bordism, by Section \ref{sec:sec2}.

 \end{remark}
 
 \begin{Question}
 Are these the natural definition of a pseudocycle in characteristic $p$? Is the vector space generated by all mod-$p$ $c$-pseudocycles up to mod-$p$ $c$-pseudocycle bordism isomorphic to the singular cohomology with $\mathbb{F}_p$-coefficients.
 \end{Question}
 
 \section{Additivity}
 \label{sec:additivity}
Given $G \le \text{Sym}(n)$, we can define some sort of operation $Q$ with inputs corresponding to orbits of the $G$-action on $\{1,\dots,n \}$ (indexed by $I$), with the space of domains determined by $A \righttoleftarrow G$. However, ideally we would like it to be additive in each of its inputs. It is definitely  additive in $C$, but the question of whether it is additive in the $I$ different tensor powers of $QH^*(M)$ is much more tricky. 

In general, it is nonobvious whether this $Q$ will be multilinear in its factors. To start with, we will prove a proposition that highly simplifies our case:

\begin{proposition}
\label{proposition:coprime-index}
If $H \le G \le \text{Sym}(n)$ and $|G:H|$ is coprime to $p$, with the $G$-action having orbits $O_1,\dots,O_I$ and the $H$-action having orbits $P_1,\dots,P_J$ of $\{1,\dots,n\}$; then if one can define additive operations with $H$, one can also define additive operations with $G$.
\end{proposition}
\begin{proof}[Sketch proof]
Observe that there is a map $\pi_*:H_*(BH;\bZ/p) \rightarrow H_*(BG;\bZ/p)$ by quotienting. As each $P_j$ is contained in some $O_i$, there is an induced map $$ \bigotimes_{i=1}^I H^*(M) \rightarrow \bigotimes_{j=1}^J H^*(M),$$ contrapositive with respect to the map $J \mapsto I$ defined by $j \mapsto i$ if $P_j \subset O_i$. Then the below map commutes for any $C$: 
\begin{equation}\label{equivpssisom}
\xymatrix{
\bigotimes_{i=1}^I H^*(M)
\ar@{->}^-{\text{pow}_G}[r]
\ar@{->}^-{\Phi}[dd]
&
H^*_{G}(C^*(M)^{\otimes n})
\ar@{->}^-{Q_j(\pi_* C')}[dr]
\ar@{->}_-{\pi^*}[dd]
&
\\ 
&
&
QH^*(M)
\\
\bigotimes_{j=1}^J H^*(M)
\ar@{->}^-{\text{pow}_{H}}[r]
&
H^*_{H}(C^*(M)^{\otimes n})
\ar@{->}_-{Q_j(|G:H|C)}[ur]
&
}
\end{equation}

Firstly, the square commutes trivially. 

The triangle commutes because $$Q_j(\pi_*(C'))(a) = Q_j(|G:H|C')(\pi^* a),$$ for any choice of $a \in \text{Im}(\text{pow}_G)$. To see this, observe that one may choose exactly the same auxiliary data to define the moduli spaces used in both sides of the equality. 

By commutativity of the diagram, plus additivity of the lower route, we observe that the upper route must also be additive. 

Further, by coprimality of $p$ to $|G:H|$, there is a lifting map $$i_* : H_*(BG;\bZ/p) \rightarrow H_*(BH;\bZ/p), \quad i_*(x) = \sum_{g_i} g_i x,$$ where $\{ g_i H \}$ is some choice of cosets of $H$ in $G$. As, $\pi_* \circ i_* = |G:H| \cdot \text{id}$, so $\tfrac{1}{|G:H|} \pi_*$ is surjective. In particular, any element $C$ of $H_*(BG;\bZ/p)$ may be lifted to $C' = i_*(C) \in H_*(BH;\bZ/p)$, such that $|G:H| C = \pi_*(C')$, so additivity holds for all such $C$.
\end{proof}

The above proposition actually simplifies our situation very well. Suppose we are working with $\bF_p$-coefficients. Then our group $G$ will contain a Sylow $p$-subgroup of order $p^j$, such that $|G|/p^j$ is coprime to $p$. In particular, Proposition \ref{proposition:coprime-index} allows us to descend to that Sylow $p$-subgroup. Henceforth in this section, we only need to consider the case where $G$ is a $p$-group.

\begin{example}
The group $\bZ/p \int \bZ/p \le \text{Sym}(p^2)$, and $p$ divides both groups to order $p+1$, hence the index is coprime to $p$. Further, both groups act transitively. Hence, in order to define additive $\bZ/p$-operations with $G=\text{Sym}(p^2)$, one only needs to consider the wreath product $\bZ/p \int \bZ/p$ (and this has additive operations via Theorem \ref{theorem:additive-covered-detective}). It is in this way that we can prove a version of the Adem Relation for quantum Steenrod $p$-th powers for all primes $p$ (see \cite[Theorem *]{wilkins18}).
\end{example}

In order to ensure additivity, we will proceed by the following, fairly standard, method: first, we recall from Theorem \ref{theorem:main-theorem} that we may extend $Q_j(C)$ so that it is actually defined on the equivariant chain complex,
$$Q_j(C): (C^*(M^{\times n}) \otimes C^*(EG))^G \rightarrow C^*(M).$$

It is worth mentioning here that, by definition, if $t \in C^*(EG)$ then \begin{equation} \label{eqn:pi-psi} \pi^* \Psi(t) = \sum_{g \in G} g \cdot t,\end{equation} recalling $\Psi$ from Section \ref{sec:stepB}.

Consider the following sufficient (but not necessarily necessary) condition on a pair $A \righttoleftarrow G$:

\begin{assumption}[Additivity Assumption on $G$]
\label{assumption:additivity-assumption}

Assume that for each closed $C \in C_i(EG \times_G A)$, there is some $q \ge 0$, and closed $T_C \in C_{i+q}(EG \times_G A)$, and $\tilde{t}_C \in C^q(EG)$ such that $d \Psi(\tilde{t}_C) = 0$, such that for each $j$:
\begin{enumerate}
    \item $Q_j(C)(\underline{x} \otimes 1) = Q_j(T_C)(\underline{x} \otimes \tilde{t}_C)$ for all $\underline{x}$,
    \item $d \Psi( \tilde{t}_C ) = 0$.
\end{enumerate}
\end{assumption}

\begin{remark}
We observe that, using the coherence equation \eqref{equation:strapping}, the given assumption amounts to the idea that given any $C \in C_i(EG \times_G A)$, there is some $T_C$ and $\tilde{t}_c$ such that $C = T_C \cap \Psi(\tilde{t}_C)$.  
\end{remark}

\begin{theorem}
Given Assumption \ref{assumption:additivity-assumption}, the map $Q$ is multilinear on the level of cohomology.
\end{theorem}
\begin{proof}
Without loss of generality, assume that $G$ is transitive (by Section \ref{subsubsec:splitting-orbits} we may treat each orbit separately). It suffices to show that for any closed $a,b \in C^*(M)$, then on cohomology $Q_j(C)((a+b)^{\otimes n} - a^{\otimes n} - b^{\otimes n})$ vanishes for all $j \ge 0$, and $C \in H_*(EG \times A)$. Observe (by abusing notation) that with the extended definition $Q_j(C)((a+b)^{\otimes n} - a^{\otimes n} - b^{\otimes n}) = Q_j(C)(((a+b)^{\otimes n} - a^{\otimes n} - b^{\otimes n}) \otimes 1),$ the latter of which by the assumption equals $Q_j(T_C)(\left((a+b)^{\otimes n} - a^{\otimes n} - b^{\otimes n}\right) \otimes \tilde{t}_C)$. Notice also that $$(a+b)^{\otimes n} - a^{\otimes n} - b^{\otimes n} = \sum_{g \in S_n} \sum_{j = 1}^{n-1} g \cdot a^{\otimes j} \otimes b^{\otimes n-j} = \sum_{g \in G} \sum_{g_i} \sum_{j = 1}^{n-1} g \cdot g_i a^{\otimes j} \otimes b^{\otimes n-j},$$ where the $G g_i$ are the right cosets of $G$ in $S_n$. We will call $\underline{x} := \sum_{g_i} \sum_{j = 1}^{n-1} g_i a^{\otimes j} \otimes b^{\otimes n-j} $.

Next, observe that $$Q_j(T_C)\left(\sum_{g \in G} g \cdot\underline{x} \otimes t_C\right) = Q_j(T_C)\left(\underline{x}  \otimes \sum_{g \in G} t_C\right) = Q_j(T_C)\left(\underline{x} \otimes \pi^* \Psi(t_C)\right).$$ Notice that $d \Psi(t_C) = 0$ hence $d \pi^* \Psi(t_C) = 0$, hence by acyclicity $\pi^* \Psi(t_C) = d m$ for some $m \in C^*(EG)$. By closedness of $a,b$  $$Q_j(T_C)((\underline{x} \otimes \pi^* \Psi(t_C)) = Q_j(T_C)(d_{eq}(\underline{x} \otimes m)).$$ By Theorem \ref{theorem:main-theorem}, the $Q_j(T_C)$ are chain maps, hence $Q_j(T_C)(d_{eq}(\underline{x} \otimes b))$ is exact.
\end{proof}

In practise, this assumption is sufficient to prove that additive operations are possible for $\bZ/2$:

\begin{example}
Suppose that $n=2$, $p=2$ and $G=\bZ/2 = \langle \sigma \rangle$. Then there is a cellular decomposition of $EG$ into discs of the form $D^{i,\pm}$ of dimension $i$, where $d D^{i,\pm} = D^{i-1,+} + D^{i-1,-}$ (neglecting signs as we work with $\bF_p$-coefficients). In this case, we can define for our additivity assumption that $Q_j(D^{i,\pm})(\underline{x} \otimes 1) = Q_j(D^{i+1,\pm})(\underline{x} \otimes (D^{1,+})^*)$. Noticing that $\Psi((D^{1,+})^*))$ has vanishing boundary, then if $A$ is fixed by $G$, this does indeed satisfy the additivity assumption and the coherence equations, hence is additive. 

Notice that the additivity assumption in this case is the requirement that $Q_j$ is ``invariant under multiplication by the degree-$1$ generator of $H^1(\bR P^{\infty}) = H^1(B \bZ /2)$".
\end{example}

For our purposes, the additivity assumption is too general to work with. However, by looking at the previous example, we see that if we assume we can make a global choice of $\tilde{t}_C$ (i.e. independent of $C$) then that is enough to prove additivity. 

\begin{assumption}[Shift Assumption]
\label{assumption:shift-assumption}
We assume that $G \subset \text{Sym}(n)$ is a group, with some $t \in H^*(BG)$ such that multiplication by $t$ is faithful on $H^*_G(A; \bF_p)$. Further, $t$ has some representative $\tilde{t} \in C^*(EG)$ (consisting of lifted cocells, and representative in the sense that $t = [\Psi(\tilde{t})]$).
\end{assumption}

Observe that there is a natural map $H^*(BG) \rightarrow H^*_G(A)$ using the associated Serre fibration, and we call the image of $t$ under this map $t_A$. We will call the faithful map $$H^*_G(A) \rightarrow H^*_G(A), \quad x \mapsto x \cup t_A,$$ by the name of the ``shift map". Notice in particular that the above assumption is independent of $n$.

We now show that Assumption \ref{assumption:shift-assumption} is indeed a stronger assumption, hence implies that additive operations can be defined.

\begin{proposition}
Assumption \ref{assumption:shift-assumption} implies Assumption \ref{assumption:additivity-assumption}.
\end{proposition}
\begin{proof}
We recall by definition that on the level of chains: $$Q_j(C')(\underline{x} \otimes \tilde{t}) := Q_j(C' \cap \Psi(\tilde{t}))(\underline{x} \otimes 1).$$ Further, every element $[C] \in H_*(B G)$ can be obtained as $[C'] \cap \Psi(\tilde{t})$ for some $C' \in H_*(BG)$ (this is because the cup product with $t$ has no kernel, and we choose $C'$ from the dual of $C \cup t$). Note also that by assumption $dt = d\Psi(\tilde{t}) =0$.

Hence, this choice is sufficient to ensure that Assumption \ref{assumption:additivity-assumption} holds.
\end{proof}

In essence, Assumption \ref{assumption:shift-assumption} is just the case of Assumption \ref{assumption:additivity-assumption} where there is some $t$ that works independently of $C$ (and such that $T_C$ is always determined by cap product of $C$ with this $t$).


The following lemma now follows:

\begin{lemma}
\label{lemma:product-of-indicators}
If $t,t'$ both demonstrate that $A \righttoleftarrow G$ satisfies Assumption \ref{assumption:shift-assumption}, then so too does $t \cup t'$.
\end{lemma}
\begin{proof}
Faithfulness of the cup product is immediate. Further, $t$ is represented by $\tilde{t}$ and $t'$ by $\tilde{t'}$ as in Assumption \ref{assumption:shift-assumption}. Note further that if one considers the chain-level definition of the simplicial cup product, that $t \cup t' = \Psi(\tilde{t} \cup \tilde{t'})$ and is closed. 
\end{proof}

\begin{example}
\label{example:z-p-additive}
Assumption \ref{assumption:shift-assumption} holds for $G = \bZ/p \subset \text{Sym}(p)$ for prime $p$, where we use $\tilde{t} \in C^2(B \bZ/p; \bF_p)$ where $\tilde{t}$ represents any generator of $H^2(B\bZ/p;\bF_p)$, and where $A$ is fixed by $\bZ/p$. Note that we do not in fact always need this assumption that $A$ is fixed: it was proven by D. Kim in \cite{kim} that one can also consider the e.g. case where $A = \overline{\mathcal{M}}_{0,1+p}$ or $A = \overline{\mathcal{M}}_{0,p}$. 
\end{example}

We now observe a way of passing Assumption \ref{assumption:shift-assumption} upwards from a subgroup. Suppose that $H \le G$. We may use $EH = EG$. In particular:

\begin{lemma}
\label{lemma:subgroup-passes-to-whole-group}
Suppose that $H \le G$, with $H$ satisfying Assumption \ref{assumption:shift-assumption} for some $A$. Then there is the map induced by the quotient map on topological spaces $\pi^*: H^*(BG;\bF_p) \rightarrow H^*(BH;\bF_p)$. Let $t \in H^*(BH;\bF_p)$ be as from Assumption \ref{assumption:shift-assumption}. If $t$ has a preimage $\hat{t} \in H^*(BG;\bF_p)$ under $\pi^*$, such that $\hat{t}$ that acts faithfully by the cup product on cohomology, then $A \righttoleftarrow G$ satisfies Assumption \ref{assumption:shift-assumption} and it is demonstrated by $\hat{t}$.
\end{lemma}
\begin{proof}[Sketch proof]
We observe that we may use the same $\tilde{t} \in C^*(EG)$ to represent both $t_{\pi^* x}$ and $\hat{t}$. The result is then immediate.
\end{proof}

\begin{corollary}
\label{corollary:coprime-index}
If $H \le G$ is of index coprime to $p$, then if $A \righttoleftarrow H$ satisfies Assumption \ref{assumption:shift-assumption} then so too does $A \righttoleftarrow G$.
\end{corollary}
\begin{proof}
The preimage of the defining $t$ is $i^*(t)$ as in Proposition \ref{proposition:coprime-index}.
\end{proof}

The lemma will allow us to prove the following:

\begin{lemma}
\label{lemma:general-n-cyclic-groups}
Assumption \ref{assumption:shift-assumption} holds for any finite Abelian group $G$, with $A = \{ pt \}$. Further, it is demonstrated by any non-nilpotent element $t \in H^{*}(BG;\bF_p)$. \end{lemma}
\begin{proof}
We first observe that an Abelian group $G$ is a direct sum of cyclic groups. First we prove the lemma for $G = \bZ/n$. From Corollary \ref{corollary:coprime-index}, we may assume that $n=p^i$. We will apply Lemma \ref{lemma:subgroup-passes-to-whole-group} with $H=\bZ /p$ and $G=\bZ/ p^i$. Observe that $\bZ/p$ satisfies the assumption, as in Example \ref{example:z-p-additive}, demonstrated by $\tilde{t}_i \in C^{2i}(E \bZ/p ;\bF_p)$ where $\tilde{t}_i$ represents any generator $t^i \in H^{2i}(B\bZ/p; \bF_p)$, which are the non-nilpotent elements of cohomology. Similarly, one notices that $H^{*}(B\bZ/p; \bF_p)$ is generated by an infinite order element of index $2$, and an element of index $1$ that squares to $0$. Picking a similarly nice cell decomposition of $E \bZ/p^i = S^{\infty} \subset \mathbb{C}^{\infty}$, we see that the generator $t \in H^2(B \bZ/p; \bF_p)$ lifts to $\tilde{t} \in C^2(E S^{\infty};\bF_p)$ which descends to $\hat{t} \in H^2(B\bZ/p^i;\bF_p)$ that is also closed. 

Suppose now that $G = \bigoplus_k \mathbb{Z}/n_k$ is a direct sum of cyclic groups. We suppose that $t$ is any non-nilpotent element of $H^*(BG;\bF_p)$. Then we can write $t = \oplus t_k$ with $t_k \in H^*(B \bZ/n_k;\bF_p)$. In particular, there is some $k$ such that $t_k$ is non-nilpotent. The result follows by Lemma \ref{lemma:subgroup-passes-to-whole-group}.
\end{proof}

Now we will use the preceding lemmas to prove that Assumption \ref{assumption:shift-assumption} holds for a large family of groups. First we recall a definition \cite[Section IV, Definition 4.2]{ademmilgram}:

\begin{definition}
Let $G \le \text{Sym}(n)$ be a finite group. Then we say $G \le \text{Sym}(n)$ has its cohomology detected by Abelian subgroups if there is a family $\{ H_i \xhookrightarrow{f_i} G \}_{i \in I}$ such that the following map is injective:

$$\bigsqcup_{i \in I} f_i^*: H^*(BG;\bF_p) \rightarrow \bigsqcup_{i \in I} H^*(BH_i;\bF_p).$$

\end{definition}

We thus prove the following additivity property for certain finite groups whose cohomology is detected by Abelian subgroups:

\begin{theorem}
\label{theorem:additive-covered-detective}
Suppose that the cohomology of $G$ is detected by Abelian subgroups, and $H^*(BG;\bF_p)$ contains some $t$ such that multiplication by $t$ is faithful. Then $\{ pt \} \righttoleftarrow G$ satisfies Assumption \ref{assumption:shift-assumption}.
\end{theorem}
\begin{proof}
We will show that the assumptions of Lemma \ref{lemma:subgroup-passes-to-whole-group} hold for some Abelian subgroup. Because the cohomology of $G$ is covered by Abelian subgroups there is some $H_i \le G$ such that $f^*_i(t)$ is non-zero. Further, as $t$ must be non-nilpotent, for each $j$ there is a $i(j)$ such that $f^*_{i(j)}(t^j) \neq 0$. As the set $\{ H_i \}$ is finite, there must be some $i$ and $j_1 < j_2 < \dots$ such that $i(j_1) = i(j_2)= \dots = i$ and such that $f^*_{i}(t^j) \neq 0$. In particular, $f_i^*(t)$ is non-nilpotent, and as $H_i$ is Abelian then by Lemma \ref{lemma:general-n-cyclic-groups} this $f_i^*(t)$ demonstrates $H_i$ satisfying Assumption \ref{assumption:shift-assumption}: further, $t$ is a preimage satisfying the assumptions of Lemma \ref{lemma:subgroup-passes-to-whole-group}, so $G$ satisfies Assumption \ref{assumption:shift-assumption} as required.
\end{proof}

\bibliographystyle{plain}
\bibliography{biblio}

\end{document}